\documentclass[11pt,a4paper]{article}
\usepackage[T1]{fontenc}
\usepackage[utf8]{inputenc}
\usepackage{lmodern}
\usepackage[margin=28mm]{geometry}
\usepackage{amsmath,amssymb,amsthm,mathtools}
\usepackage{microtype}
\usepackage{enumitem,booktabs,array}
\usepackage{xcolor}
\usepackage[colorlinks=true,linkcolor=blue!75!black,citecolor=blue!75!black,urlcolor=blue!75!black]{hyperref}
\hypersetup{pdftitle={Curvature-Dimension Conditions Along Discrete Refinement Sequences},pdfauthor={C.TAO}}
\numberwithin{equation}{section}
\newtheorem{theorem}{Theorem}[section]
\newtheorem{proposition}[theorem]{Proposition}
\newtheorem{lemma}[theorem]{Lemma}
\newtheorem{corollary}[theorem]{Corollary}
\theoremstyle{definition}
\newtheorem{definition}[theorem]{Definition}

\newtheorem{example}[theorem]{Example}

\theoremstyle{remark}
\newtheorem{remark}[theorem]{Remark}
\newcommand{\R}{\mathbb R}

\newcommand{\Pp}{\mathcal P}
\newcommand{\dd}{\mathrm d}
\newcommand{\eps}{\varepsilon}

\newcommand{\SCD}{\mathrm{SCD}}
\newcommand{\SCDmid}{\mathrm{SCD}_{\mathrm{mid}}}
\newcommand{\CDe}{\mathrm{CD}^{e}}
\newcommand{\ruleop}{\mathcal R}

\DeclareMathOperator{\Ent}{Ent}
\DeclareMathOperator{\supp}{supp}
\DeclareMathOperator{\diam}{diam}
\DeclareMathOperator{\Lip}{Lip}

\setlist[enumerate]{leftmargin=*,itemsep=3pt,topsep=5pt}
\setlist[itemize]{leftmargin=*,itemsep=3pt,topsep=5pt}
\title{CURVATURE--DIMENSION CONDITIONS ALONG DISCRETE REFINEMENT SEQUENCES\\[0.35em]}

\author{CHEN TAO
\thanks{%
School of Mathematics and Statistics,
Northeast Normal University,
Changchun, Jilin 130024, China.
E-mail: \href{mailto:taoc387@nenu.edu.cn}%
{\texttt{taoc387@nenu.edu.cn}}.}}
\date{}

\begin{document}

\maketitle
\begin{abstract}
We introduce the sequence curvature--dimension condition
$\SCD(K,N)$ for complete refinement systems of finite metric
probability spaces.
The condition couples a transport action bound with entropy
comparisons along a single probability section through the
prescribed refinement layers.
Within a bounded comparison range, we prove that midpoint
comparisons on three consecutive layers characterize the condition
on every fixed finite segment.
If, in addition, the refinement sequence converges in the measured
sense in its intrinsic bridge realization, then $\SCD(K,N)$
is equivalent both to the vanishing of single-layer midpoint defects
and to the entropic curvature--dimension condition $\CDe(K,N)$
of the limit.
In the reverse direction, a $\CDe(K,N)$ bound on the limit
yields quantitative recovery estimates that identify shrinking
transport scales on which finite layers retain its second-order
curvature comparison.
These estimates are uniform over endpoint measures
and include Dirac masses.
Examples illustrate the effects of the refinement rule,
the layer metrics and the reference measures.
\end{abstract}
\tableofcontents
\clearpage


\section*{Introduction}
\phantomsection
\addcontentsline{toc}{section}{Introduction}
\label{sec:introduction}

Synthetic lower Ricci bounds express curvature through entropy
comparisons along optimal transport. McCann's displacement convexity
\cite{McCann1997} and the curvature--dimension theories of Lott--Villani
and Sturm \cite{LV2009,Sturm2006I,Sturm2006II} provide a framework
that is stable under measured convergence. On a finite discrete space,
however, Dirac measures at two nearest distinct points admit no
midpoint for the quadratic Wasserstein distance. 
One way to accommodate this obstruction is to allow approximate
interpolation at a prescribed spatial scale.
This approach underlies the rough curvature bounds of
Bonciocat--Sturm \cite{BonciocatSturm2009}
and Bonnefont's approximate Brunn--Minkowski inequality
\cite{Bonnefont2009}.
Bonciocat \cite{Bonciocat2014} develops a finite-dimensional
rough curvature--dimension condition with stability and
discretization theorems.
Kazukawa, Ozawa and Suzuki \cite{KazukawaOzawaSuzuki2020}
prove further stability results for rough curvature--dimension
conditions.

We introduce a sequence curvature--dimension condition $\SCD(K,N)$ 
for refinement systems.
Refinement is a basic tool for studying discrete spaces.
A refinement sequence of a given discrete object may reveal 
some hidden geometric features at successive resolutions,
while the connections between successive refinement layers provide an additional direction
for geometric comparison.
Unlike the classical or rough curvature--dimension conditions,  
which consider comparisons within a fixed metric measure space,
we regard the entire sequence generated by a given refinement rule 
as a single geometric object,
and study transport and entropy comparisons across successive layers.
On each fixed finite segment, a single sequence of intermediate
measures must satisfy both the transport action bound and the
entropy comparisons at every interior layer.

In the present paper, we establish three main results.
First, midpoint comparisons on 
three consecutive refinement layers characterize
the full condition on every fixed finite segment, 
without assuming convergence.
Under measured convergence in the intrinsic bridge realization,
we then prove that $\SCD(K,N)$ is equivalent to the entropic
curvature--dimension condition $\CDe(K,N)$ on the limit.
Furthermore, 
if the limit satisfies this condition, 
quantitative recovery estimates identify shrinking transport scales 
on which finite layers retain the second-order curvature comparison of the limit.
Together, these results connect entropy comparisons along the
refinement sequence with curvature bounds on the limit and their
quantitative recovery on finite layers.

Other approaches construct interpolations on a fixed graph:
Gozlan, Roberto, Samson and Tetali \cite{GozlanRobertoSamsonTetali2014}
study displacement convexity of entropy; L\'eonard \cite{Leonard2016}
uses slowed reference random walks; and Samson \cite{Samson2022}
studies entropic curvature along zero-temperature Schr\"odinger bridges.
The transport metrics of Maas, Erbar--Maas and Mielke
\cite{Maas2011,ErbarMaas2012,Mielke2013} are adapted to reversible
Markov chains and their entropy gradient flows. Ollivier's curvature
\cite{Ollivier2009} uses contraction of transition probabilities in
$W_1$.

For modified discrete transport metrics, Gigli--Maas
\cite{GigliMaas2013} prove convergence from discrete tori to the
continuous torus, while Gladbach, Kopfer and Maas
\cite{GladbachKopferMaas2020} identify the role of mesh isotropy
in discretizations of Euclidean domains.
Our comparisons use the usual metric Wasserstein distance,
with the layer metrics and the prescribed bridges determining
the transport geometry.

We now fix the basic setting and notation for the theory.
Let $\Sigma$ be a complete refinement system as in
Section~\ref{sec:refinement}, 
with finite layers $X_n=(V_n,d_n,m_n)$ 
and intrinsic bridge realization $(Z,d_Z)$.
For some sufficiently large $R$, assume that
$\bigcup_{n\ge R}V_n$ has diameter at most $L>0$ in $Z$.
Fix $K \in\mathbb R$ and $1 \le N < \infty$, 
with $ KL^{2} < N\pi^{2}$ when $K>0$, 
and write $W=W_{2,Z}$ and $\kappa=K/N$.

\paragraph{Midpoint characterization.}
A sequence $\boldsymbol\mu=(\mu_0,\ldots,\mu_J)$ with
$\mu_j\in\Pp(V_{r+j})$ is called an \emph{$\eta$-admissible probability section}
if its transport action is at most $D^2+\eta L^2$, 
where $D=W(\mu_0,\mu_J)$ \eqref{eq:section-action}.
This bound expresses approximate constant-speed transport
through the prescribed refinement layers. 
For every fixed $J\ge2$ and $\eta>0$, complete refinement
ensures the existence of such sections on all sufficiently late
segments, uniformly over the endpoints
(Theorem~\ref{thm:admissible-existence}).

Let $H_n$ be the relative entropy with respect to $m_n$ and
write $U_n=e^{-H_{n}/N}$.
Following the entropic formulation of Erbar--Kuwada--Sturm
\cite[Definition~3.1]{EKS2015}, 
the condition $\SCD(K,N)$ requires, 
for every fixed $J\ge2$ and $\eps>0$, 
that each endpoint pair on every sufficiently late segment admit one
$\eps$-admissible section satisfying
\[
 U_{r+j}(\mu_j)\ge
 \sigma_\kappa^{(1-j/J)}(D)U_r(\mu_0)
 +\sigma_\kappa^{(j/J)}(D)U_{r+J}(\mu_J)-\eps,
 \qquad 1\le j<J.
\]
Here $\sigma_\kappa^{(t)}$ is the comparison coefficient defined
in Section~\ref{sec:entropy}.
The case $J=2$ gives the midpoint condition $\SCDmid(K,N)$.

\theoremstyle{plain}
\newtheorem{introtheorem}{Theorem}
\renewcommand{\theintrotheorem}{\Alph{introtheorem}}

\begin{introtheorem}[Midpoint characterization]
\label{intro:midpoint}
Under the standing assumptions,
\[
 \Sigma\text{ satisfies }\SCD(K,N)
 \quad\Longleftrightarrow\quad
 \Sigma\text{ satisfies }\SCDmid(K,N).
\]
\end{introtheorem}

This theorem requires no convergence assumption.
The main point is to obtain all interior entropy comparisons
on the same section while preserving its admissibility.
The proof uses successive midpoint replacements to control
the entropy comparison deficits and the accumulated action error;
see Theorem~\ref{thm:direct-midpoint}.

\paragraph{Limit characterization.}
Suppose now that the layers converge 
in their intrinsic bridge realization:
\[
 d_H^Z(V_n,M)\longrightarrow0,\qquad
 m_n\rightharpoonup m,\qquad \supp m=M,
\]
where $M\subset Z$ is nonempty and compact, and $d=d_Z|_{M\times M}$.
This is the measured Gromov--Hausdorff 
convergence of Definition~\ref{def:convergence},
which retains compatibility with the prescribed bridges.
Let $\Delta_{n}(K,N)$ be the single-layer midpoint defect from
\eqref{eq:single-defect}.
It is the least common tolerance for which every pair of
probability measures on $V_n$ admits one intermediate measure
on the same layer satisfying both the midpoint action bound
and the entropy comparison.

\begin{introtheorem}[Limit characterization]
\label{intro:limit}
Under the standing assumptions and the convergence above,
\begin{gather*}
 \SCD(K,N)
 \Longleftrightarrow
 \Delta_n(K,N)
 \rightarrow 0
 \Longleftrightarrow
 (M,d,m)\text{ satisfies }\CDe(K,N).
\end{gather*}
\end{introtheorem}

This equivalence connects comparisons across refinement layers,
vanishing midpoint defects within individual layers, and
entropy comparisons along Wasserstein geodesics on the limit.
In the reverse direction, the limit condition recovers the
discrete comparisons uniformly over endpoint measures,
which may vary with the refinement level.
For comparisons across layers, the segment length is fixed;
see Theorem~\ref{thm:characterization}.

If $N>1$ and the limit is essentially nonbranching, the equivalent
conditions also include the classical $\mathrm{CD}(K,N)$ condition.
This follows through the reduced condition $\mathrm{CD}^{*}(K,N)$
of Bacher--Sturm \cite{BacherSturm2010}, using
\cite[Theorem~3.12]{EKS2015} and
\cite[Main Theorem~1.1]{CavallettiMilman2021}; see
Corollary~\ref{cor:classical-cd}.

\paragraph{Recovery at shrinking transport scales.}
To retain the second-order curvature comparison 
at shrinking transport distances, 
we estimate the recovery errors relative
to the squared endpoint distance.
Assume in addition that the limit satisfies $\CDe(K,N)$.
Set $h_n:=W_{\infty,Z}(m_n,m)$.
The convergence assumptions of Theorem~\ref{intro:limit},
together with the refinement geometry, imply $h_n\to0$
(Lemma~\ref{lem:upgrade}).
Write $W_n=W_{2,(V_n,d_n)}$ and
$S_\kappa(D)=\sigma_\kappa^{(1/2)}(D)$.
The quantitative result controls both comparisons at the same
approximate midpoint.

\begin{introtheorem}[Single-layer midpoint recovery]
\label{intro:recovery}
For every $n\ge R$ and $p,q\in\Pp(V_n)$ with $D=W_n(p,q)>0$,
there is $\nu\in\Pp(V_n)$ such that, with
$a=W_n(p,\nu)$ and $b=W_n(\nu,q)$,
\begin{align*}
 0\le2a^2+2b^2-D^2&\le12Dh_n+36h_n^2,\\
 \frac{U_n(\nu)}{U_n(p)+U_n(q)}
 &\ge S_\kappa(D)-2C_{\kappa,L}h_n(D+h_n).
\end{align*}
The constant $C_{\kappa,L}$ is the derivative bound in
Theorem~\ref{thm:single-recovery} and is independent of the endpoints.
For $D=0$, taking $\nu=p=q$ gives exact comparisons.
\end{introtheorem}

The factor $h_n(D+h_n)$ comes from $S_\kappa'(0)=0$ and retains
the second-order curvature term. The action excess and normalized
entropy comparison deficit, each divided by $D^2$, tend uniformly
to zero for $D\ge r_n$ whenever $h_n=o(r_n)$.
In particular, the expansion
$S_\kappa(D)=1/2+KD^2/(16N)+O_\kappa(D^4)$ gives
\[
 \frac{U_n(\nu)/(U_n(p)+U_n(q))-1/2}{D^2}
 \ge \frac{K}{16N}-o(1)
\]
uniformly for $0<r_n\le D\le R_n\to0$ with $h_n/r_n\to0$.
This is a lower comparison at the fixed dimension $N$.
It includes Dirac endpoints and requires no uniform entropy bound.
On uniform interval grids, nearest-neighbor Dirac pairs satisfy
$D=2h_n$ and have a minimum relative action defect of one
(Corollary~\ref{prop:nearest-scale}).
Thus relative action errors need not vanish at transport
distances comparable to $h_n$.

Bounds on $h_n$ in terms of geometric approximation and reference
mass discrepancy give explicit conditions for these comparison scales.
We also study parameter monotonicity, invariance, and finite comparison
defects on segments. Within the comparison range, the measured limit
determines the admissible parameter region, while the finite defects
and their rates retain information about the chosen refinement sequence.

The examples make this dependence concrete. Weighted adaptive
intervals have an explicit curvature range, and two refinements of
the same initial hypercube graph show the effect of the layer metrics.
Further examples distinguish complete resolution, endpoint control,
and entropy comparison; the specified Goldberg--Coxeter sequence
exhibits a midpoint obstruction directly on its layers.

Sections~\ref{sec:refinement}--\ref{sec:direct-midpoint} develop
the framework and prove the midpoint characterization.
Sections~\ref{sec:limits}--\ref{sec:characterization} establish
convergence criteria and the limit characterization.
Sections~\ref{sec:quantitative-recovery}--\ref{sec:invariance}
give the quantitative estimates and invariance results, and
Section~\ref{sec:examples} contains the examples.

\clearpage
\section{Discrete spaces and local refinement}
\label{sec:refinement}

We begin with finite discrete spaces 
and the local rules that generate successive refinement layers. 
We then define refinement systems and 
show how bridges between adjacent layers determine a common metric realization.

\subsection{Discrete spaces and generating rules}

\begin{definition}[Finite discrete space]
\label{def:discrete-space}
A finite discrete space is a triple $X=(V,d,m)$, 
where $V$ is a nonempty finite set, 
$d$ is a metric on $V$, 
and $m$ is a probability measure 
satisfying $m(x)>0$ for every $x\in V$.
\end{definition}

An element of $V$ may represent a vertex, a cell, or another discrete object. 
Thus the definition accommodates finite graphs, point configurations, 
and cell systems, without prescribing a dimension or a product structure. 
For a region with fixed exterior data, 
we require mass transfer to remain within the region 
and preserve its total reference mass.

To describe local refinement, 
we equip $X$ with the finite geometric data required by the chosen rule. 
These may include 
incidence relations, adjacency relations, local scales, cell boundaries, interface marks, 
or prescribed boundary values. 
Interfaces may be recorded among these data 
without being included in $V$ or assigned reference mass. 
For example, when $V$ consists of the top-dimensional cells of a cell complex, 
reference mass is assigned to these cells, 
while their shared boundaries enter only as geometric data.
We write $\mathcal X$ for $X$ together with these data.

\begin{definition}[Stationary local refinement rule]
\label{def:local-rule}
A stationary local refinement rule $\ruleop$ 
maps an admissible space $\mathcal X_n$, equipped with its geometric data, 
to $\mathcal X_{n+1}$ by a procedure satisfying the following requirements.
\begin{enumerate}
\item Each new local piece depends on a neighborhood of uniformly bounded radius 
         in the adjacency structure of $\mathcal X_n$, 
         together with the prescribed boundary or collar data. 
         Any required bound on the number of elements in a geometric stencil 
         is included in the admissible class.
\item The same procedure is used at each refinement step, 
         with the current scales and marks as inputs.
\item The rule specifies the new adjacency relations and
         the lengths of the corresponding local connections,
         and determines the new metric by shortest paths.
         The rule also specifies compatible shared interfaces,
         the new reference measure, and the comparison data introduced below.
\item The rule commutes with isomorphisms of metric probability spaces
         that preserve the specified geometric data. 
         When the construction admits several outcomes, 
         a choice of outcome or branch is part of the refinement system.
\end{enumerate}
\end{definition}

We call each $X_n$ a \emph{refinement layer}, or simply a \emph{layer};
$\mathcal X_n$ denotes the same layer equipped with its geometric data. 
Each layer thus carries the metric naturally determined 
by its local connections and their lengths.
Since refinement may change the adjacency relations,
distances are computed using the connections of the new layer.
We adopt the same convention for the initial layer.

Stationarity means that the procedure is independent of $n$, 
although the data to which it applies may vary. 
This agrees with the use of stationarity for subdivision schemes; 
see Dyn and Levin \cite{DynLevin2002}.

The locality bound is uniform over the layers; 
finiteness of each layer alone gives no such bound. 
In a geometric model, 
locality in the adjacency structure does not imply 
that the corresponding geometric neighborhoods shrink: 
their diameters must also tend to zero for this stronger property to hold. 
Stationarity allows the scales and displacements to vary. 
Shared interfaces may be constructed simultaneously from the preceding layer 
when the rule ensures compatibility. 
We do not require the result to be independent 
of the order of successive local transformations.

An adaptive choice of regions may leave some patches unchanged 
at a given refinement step. 
The same local procedure is used throughout,
with this choice recorded among the data when needed. 
To call the process a refinement, 
we also require the resolution improvement defined below. 
A relation between consecutive layers alone could equally describe mixing or coarsening.

\subsection{Interlayer mass transport and bridges}

\begin{definition}[Comparison data]
\label{def:bridge}
Let $X_n$ and $X_{n+1}$ be consecutive layers. 
A bridge $b_n$ is a pseudometric on the disjoint union $V_n\sqcup V_{n+1}$ 
whose restrictions are $d_n$ and $d_{n+1}$. 
A reference coupling is a nonnegative array
$\gamma_n\in\Pi(m_n,m_{n+1})$, 
that is,
\begin{equation}
 \sum_y\gamma_n(x,y)=m_n(x),\qquad
 \sum_x\gamma_n(x,y)=m_{n+1}(y).
 \label{eq:reference-marginals}
\end{equation}
A comparison system $\Sigma$ records the layers with 
their geometric data, the bridges, and the reference couplings. 
A refinement system also specifies how these data are generated; 
its full definition is given below. 
\end{definition}

Elements from different layers may have zero bridge distance.
The relation $\gamma_n(x,y)>0$ records mass transfer 
from $x\in V_n$ to $y\in V_{n+1}$. 
A source element may supply several target elements,
and a target element may receive mass from several source elements.
Geometric dependence is recorded separately. 
A neighboring cell, for instance, 
may constrain the position or interface of a new element without supplying its mass.

The reference coupling gives forward and backward stochastic kernels
\begin{equation}
 P_n(x,y)=\frac{\gamma_n(x,y)}{m_n(x)},\qquad
 Q_n(y,x)=\frac{\gamma_n(x,y)}{m_{n+1}(y)}.
 \label{eq:reference-kernels}
\end{equation}
Thus 
$m_{n}P_{n}(y) = m_{n+1}(y)$ and $m_{n+1}Q_{n}(x) = m_n(x)$. 
These kernels describe the prescribed transfer of reference mass. 
The curvature tests will allow all couplings permitted by the bridge metric, 
without restricting them to the support of the reference coupling.

For fixed consecutive layers and a cost $c_n(x,y)$,
the reference coupling is optimal if
\begin{equation}
 \gamma_n\in
 \operatorname*{argmin}_{\gamma\in\Pi(m_n,m_{n+1})}
 \sum_{x,y}c_n(x,y)\gamma(x,y).
 \label{eq:reference-optimality}
\end{equation}
The coupling polytope is nonempty and compact, 
so continuity of the objective gives a minimizer. 
For the bridge geometry, a natural choice is
$c_n=b_n^2$. If the generating rule restricts the feasible set,
optimality is understood relative to that set.

\subsection{The refinement system}

\begin{definition}[Geometric resolution data]
\label{def:resolution}
Geometric resolution data consist of:
\begin{enumerate}
\item a local geometric scale $s_n(x) \ge 0$ for each $x\in V_n$,
determined by the geometric data of $\mathcal X_n$;
\item a nonempty dependency set $G_n(y)\subset V_n$
for each $y\in V_{n+1}$, consisting of the elements used
by the local rule to construct the geometry of $y$.
\end{enumerate}
\end{definition}

The local geometric scale measures both the size of the element
and its spacing from adjacent elements.
To be more precise, 
let $\mathcal N_n(x)$ denote the elements adjacent to $x$
under the adjacency relation specified by the refinement rule,
let $r_n(x)\ge0$ denote the geometric size of the element
represented by $x$, computed from its local geometry
using the same prescription at every layer.
For instance, a point element may have size of zero and 
a cell element have the size of its diameter
or radius about a designated representative.
We define
\begin{equation}
 s_n(x) :=
 \max\left\{
  r_n(x),\
  \max_{z\in\mathcal N_n(x)}d_n(x,z)
 \right\},
 \label{eq:local-resolution-scale}
\end{equation}
where the maximum over an empty neighborhood is taken to be zero.
On the other hand, 
geometric dependence is distinct from mass transfer.
An element $x\in G_n(y)$ may influence the position of $y$
without contributing mass to it.
Mass transfer is recorded separately by $\gamma_n(x,y)>0$.

\begin{definition}[Local refinement and complete resolution]
\label{def:complete-refinement}
A local transformation is a refinement step if each new element
$y\in V_{n+1}$ is either an unchanged copy of an element $x\in V_n$,
with $b_n(x,y)=0$ and $s_{n+1}(y)=s_n(x)$, or a replacement whose
associated patch has finer geometric scale in the sense that
\begin{equation}
 s_{n+1}(y) \le \max_{x\in G_n(y)}s_n(x).
 \label{eq:local-resolution}
\end{equation}
A sequence of refinement steps is a complete refinement of the
selected finite region if
\begin{equation}
 \ell_n := \max_{x\in V_n}s_n(x) \longrightarrow 0.
 \label{eq:complete-resolution}
\end{equation}
\end{definition}

The local refinement condition makes $\ell_n$ nonincreasing.
Equality may occur when some local geometry is refined
without reducing the maximal scale.
Completeness requires both the sizes of the elements
and the distances between adjacent elements
to tend uniformly to zero.

The maximum over $G_n(y)$ allows the new geometry to depend on several
elements of the preceding layer, without requiring containment in any
one of their patches. An additional condition $s_{n+1}(y)<s_n(x)$ for
every $x$ with $\gamma_n(x,y)>0$ will be imposed only when a result
needs it. Neither the sets $V_n$ nor all successive patches need be
nested. Their cardinalities may increase, but do not by themselves
measure resolution.

\begin{definition}[Refinement system]
\label{def:refinement-system}
A refinement system consists of an admissible class of finite discrete
spaces equipped with geometric data, a stationary local rule
$\ruleop$, an initial layer $\mathcal X_0$, and successive layers
$\mathcal X_{n+1}=\ruleop(\mathcal X_n)$. 
It records the local geometric scale
$s_n$ on each layer and, for each refinement step, a bridge $b_n$,
a reference coupling $\gamma_n$, and geometric dependency set $G_n$,
subject to the following compatibility requirements.
\begin{enumerate}
\item The triple $X_n=(V_n,d_n,m_n)$ is the metric probability space
underlying $\mathcal X_n$. Its scales, adjacency relations, and
interface data are those specified in $\mathcal X_n$.
\item The bridge and reference coupling are specified by the
generating rule and satisfy Definition~\ref{def:bridge}.
\item Each $G_n(y)$ lies in the uniformly local neighborhood used to
generate $y$, and every step is a local refinement in the sense of
Definition~\ref{def:complete-refinement}.
\end{enumerate}
\end{definition}

We write $\Sigma$ for the system and, 
when only comparison data enter a statement, 
for its underlying comparison system. 
The generating rule may determine the new geometry 
and reference masses through a joint local optimization. 
Its admissible candidates, 
boundary and mass constraints, 
and objective must first be specified. 
Optimal transport can then be used to compare the candidates. 
The resulting construction must satisfy the resolution condition in
Definition~\ref{def:complete-refinement};
nondegeneration must also be checked.

\begin{example}
\emph{Barycentric subdivision.}
Let $V_n$ consist of the top-dimensional simplices of a finite pure
$k$-dimensional simplicial complex embedded in Euclidean space,
where $k\ge1$. 
Use barycentre distances of adjacent simplices for $d_n$, normalized
$k$-dimensional volumes for $m_n$ 
and define $s_{n}$ by \eqref{eq:local-resolution-scale}.
Barycentric subdivision divides each simplex into smaller simplices;
their geometric dependency sets consist of the containing simplex in the
preceding layer. 
Its mass is distributed according to their volumes and 
bridges are chosen to be compatible with the layer metrics.
The local geometric scale bound
$\ell_{n+1}\le \frac{k}{k+1}\ell_n$ gives complete refinement;
see Hatcher \cite{Hatcher2002}.
\end{example}

\begin{example}
\emph{Goldberg--Coxeter subdivision.}
For finite trivalent discrete surfaces, the $\mathrm{GC}_{2,0}$ construction
combines combinatorial subdivision with a local Dirichlet problem
and a subsequent displacement of the old vertices; see Kotani,
Naito and Tao \cite{KotaniNaitoTao2022}. Several elements of the preceding layer
determine the new geometry, and the realized vertex sets need not
be nested. This illustrates the geometric flexibility allowed here.
To obtain a refinement system, one must also specify the reference
measures and comparison data and verify the resolution conditions.
\end{example}
\subsection{The intrinsic bridge realization}

The individual triples $X_n$ do not determine distances between layers.
The bridges supply these distances for consecutive layers. A finite
segment of the refinement sequence consists of consecutive layers
$X_r,\ldots,X_{r+J}$, together with their associated data. It comprises
$J$ refinement steps and $J+1$ layers. We now show that the bridges
determine compatible metrics on these segments, and hence a common
metric realization of the whole sequence.

\begin{lemma}[Isometric gluing and shortest chains]
\label{lem:gluing}
The bridges admit a common metric realization, after identifying
zero-distance points, in which every layer and every bridge embeds
isometrically. Let $Z$ be its completion. For integers $r,J\ge0$,
distances between points belonging to the layers $X_r,\ldots,X_{r+J}$
agree with the metric $D_{[r,r+J]}$ obtained by gluing only the bridges
$b_r,\ldots,b_{r+J-1}$ and identifying zero-distance points.
For $J=0$, set $D_{[r,r]}=d_r$. For $J\ge1$, $x_0\in V_r$, and
$x_J\in V_{r+J}$,
\begin{equation}
 d_Z(x_0,x_J)=
 \min_{x_j\in V_{r+j},\ 1\le j<J}
 \sum_{j=0}^{J-1}b_{r+j}(x_j,x_{j+1}).
 \label{eq:monotone-chain}
\end{equation}
\end{lemma}
\begin{proof}
First glue two finite pseudometric spaces $A$ and $B$ along a common
nonempty isometric subset $S$. For $x\in A$ and $y\in B$, set
\[
 d(x,y)=\min_{z\in S}\{d_A(x,z)+d_B(z,y)\},
\]
while retaining the distances within each piece. This is the infimum
of the lengths of chains in the two pieces. To see that the original
distances are preserved, suppose a chain leaves $A$ at $z\in S$ and
returns at $w\in S$. The portion in $B$ has length at least
$d_B(z,w)=d_A(z,w)$ and can be replaced by a step within $A$ of no
greater length. The same argument applies to $B$. Concatenation of
chains gives the triangle inequality. Taking the quotient by zero
distances gives the metric gluing.

Apply this construction successively to $b_r,\ldots,b_{r+J-1}$ for
$J\ge1$. Every layer and bridge remains isometric by induction.
Adding a bridge at either end leaves the existing distances unchanged:
any portion of a chain outside the segment returns through the same
endpoint layer and can be replaced by a step within that layer.
Thus the pseudometrics obtained from finite segments agree on their
common layers. They define a pseudometric on the disjoint union of
all layers; its quotient by zero distances has completion $Z$.

To prove the chain formula, take a chain from $x_0\in V_r$ to
$x_J\in V_{r+J}$. Since each bridge joins consecutive layers, the
chain visits every intermediate layer. Denote its successive first
visits to these layers by $x_1,\ldots,x_{J-1}$. The portion between
$x_j$ and $x_{j+1}$ has length at least
$b_{r+j}(x_j,x_{j+1})$, because that bridge is isometric in the
realization. Summing proves the lower bound in
\eqref{eq:monotone-chain}. Conversely, any choice of the intermediate
$x_j$ gives a chain of the displayed length. This proves the upper
bound. Finiteness makes the infimum a minimum, and completion preserves
the resulting distances.
\end{proof}

We call $d_Z$ the \emph{intrinsic bridge metric}: 
``intrinsic'' refers to its determination by the prescribed bridges. 
The space $Z$ need not be a length space, 
and the shortest chains in \eqref{eq:monotone-chain} are finite lists of points. 
They need not follow the support of $\gamma_n$. 
The intrinsic bridge metric may differ from the distance 
induced by a common ambient space 
in which the entire refinement sequence is realized.

The local geometric scales describe the fineness of each layer.
We also use them to control the displacement between consecutive layers.
For the entropy comparisons below, we consider complete refinements
and require that there exist a constant $C$ such that,
for all sufficiently large $n$,
\begin{equation}
 \gamma_n(x,y)>0
 \quad\Longrightarrow\quad
 b_n(x,y)\le C \ell_n,
 \qquad x\in V_n,\quad y\in V_{n+1}.
 \label{eq:reference-displacement}
\end{equation}

Since the reference measures have full support,
the marginal identities \eqref{eq:reference-marginals}
ensure that every element of either layer is paired,
with positive reference mass, with an element of the other layer.
The bridges embed isometrically in $Z$, so
\eqref{eq:reference-displacement} gives
\begin{equation}
 d_{\mathrm H}^{Z}(V_n,V_{n+1}) \le C \ell_n \longrightarrow 0.
 \label{eq:adjacent-layer-distance}
\end{equation}
Thus consecutive layers become arbitrarily close
in the intrinsic bridge realization.

\begin{remark}
Condition \eqref{eq:reference-displacement} requires elements
exchanging reference mass between consecutive layers
to remain within a distance controlled by the geometric scale
$\ell_n$.
This is a natural requirement for local refinement:
mass is redistributed at the scale of the geometry being refined.
In what follows, when no confusion can arise,
we shall refer to a complete refinement satisfying
\eqref{eq:reference-displacement}
simply as a complete refinement.
\end{remark}

\section{Entropy and the SCD condition}
\label{sec:entropy}\label{sec:scd}

We first define relative entropy on the refinement layers
and establish its contraction under reference transport.
A complete refinement  provides 
probability sections that approximate constant-speed transport
through successive layers. 
We call them the admissible probability sections 
and then prove their existence on sufficiently late segments of each fixed length.
At the end of this section, 
we define the SCD condition as an entropy comparison along such sections.

\subsection{Relative entropy and reference transport}

Fix $1 \le N <\infty$. For a reference probability measure $m$, define
the relative entropy and its exponential transform by
\begin{equation}
 \Ent_m(\mu)=\begin{cases}
 \displaystyle\int\rho\log\rho\,\dd m,&\mu=\rho m,\\
 +\infty,&\mu\not\ll m,
 \end{cases}
 \qquad U_N(\mu\mid m)=e^{-\Ent_m(\mu)/N}.
 \label{eq:entropy-power}
\end{equation}
We use the conventions $0\log0=0$ and $e^{-\infty}=0$.
Jensen's inequality gives $\Ent_m(\mu)\ge0$ and
$0\le U_N(\mu\mid m)\le1$. 
On the finite layer $X_n$, 
we write
\begin{equation}
 H_n(p)=\sum_{x\in V_n}p(x)\log\frac{p(x)}{m_n(x)},
 \qquad U_n(p)=e^{-H_n(p)/N}>0.
 \label{eq:finite-entropy}
\end{equation}
Here $p\in\Pp(V_n)$, 
and $U_n$ always refers to the fixed parameter $N$. 
The strict positivity follows from the full support of $m_n$.

The reference kernels of Section~\ref{sec:refinement} 
transport densities by conditional averaging. 
We first consider the corresponding entropy contraction. 
Let $X=(V,d,m)$ and $X'=(V',d',m')$ 
be two adjacent refinement layers, 
and let $\gamma\in\Pi(m,m')$.
Set
\[
 P(x,y)=\frac{\gamma(x,y)}{m(x)},\qquad
 p=\rho m,\qquad p'=pP=\rho'm',\qquad
 \rho'(y)=\frac{\sum_x\gamma(x,y)\rho(x)}{m'(y)}.
\]
On the support $A=\{(x,y):\gamma(x,y)>0\}$, define the probability
measures
\[
 \widehat m(x,y)=\gamma(x,y),\qquad
 \widehat p(x,y)=\rho(x)\gamma(x,y).
\]
Thus $\widehat p$ has marginals $p$ and $p'$.

\begin{proposition}[Entropy contraction under reference transport]
\label{prop:carrier-entropy}\label{lem:markov-entropy}
With the notation above,
\begin{equation}
 \Ent_{\widehat m}(\widehat p)=\Ent_m(p),\qquad
 \Ent_m(p)-\Ent_{m'}(p')=
 \sum_{x,y}\gamma(x,y)\rho(x)\log\frac{\rho(x)}{\rho'(y)}\ge0.
 \label{eq:carrier-entropy}
\end{equation}
Terms with $\gamma(x,y)\rho(x)=0$ are understood to be zero.
Equality in the entropy inequality holds if and only if, for each
$y\in V'$, the density $\rho$ is constant on
$\{x\in V:\gamma(x,y)>0\}$.

More generally, let $m,m'$ be reference probability measures on
standard Borel spaces, and let $P$ be a Markov kernel with $mP=m'$.
Then, for every probability measure $\mu\ll m$,
\begin{equation}
 \Ent_{m'}(\mu P)\le\Ent_m(\mu).
 \label{eq:kernel-entropy-contraction}
\end{equation}
\end{proposition}
\begin{proof}
We first prove the kernel inequality. Define
$\Gamma(\dd x,\dd y)=m(\dd x)P(x,\dd y)$. If $\mu=\rho m$, then
$\mu P=\rho'm'$, where
$\rho'(y)=\mathbb E_\Gamma[\rho(x)\mid y]$. Conditional Jensen gives
\[
 \rho'(y)\log\rho'(y)
 \le\mathbb E_\Gamma[\rho(x)\log\rho(x)\mid y].
\]
Integration yields \eqref{eq:kernel-entropy-contraction}. Since
$u\log u$ is bounded below, the argument also applies when the
input entropy is infinite.

In the finite case, $\Gamma=\gamma$, and the marginal identities give
\begin{align*}
 \Ent_m(p)&=\sum_{x,y}\gamma(x,y)\rho(x)\log\rho(x),\\
 \Ent_{m'}(p')&=\sum_y m'(y)\rho'(y)\log\rho'(y)
 =\sum_{x,y}\gamma(x,y)\rho(x)\log\rho'(y).
\end{align*}
If $\rho'(y)=0$, then $\rho(x)=0$ for every $x$ with
$\gamma(x,y)>0$, so the corresponding terms vanish. 
Subtraction therefore gives the entropy difference in
\eqref{eq:carrier-entropy}. 
Strict convexity of $u\log u$ shows that equality holds precisely 
when $\rho$ is constant on each of the stated sets. 
Finally, the density of $\widehat p$ with respect to
$\widehat m$ is $\rho(x)$, whence
\[
 \Ent_{\widehat m}(\widehat p)
 =\sum_{x,y}\gamma(x,y)\rho(x)\log\rho(x)
 =\Ent_m(p).
\]
\end{proof}

This contraction is a consequence of averaging and does not involve
the metric. 
The curvature condition will compare entropy for arbitrary endpoint probabilities through transport in $(Z,d_Z)$. 
In these comparisons, 
all transport plans with the prescribed marginals are
admissible; their supports need not lie in the support of the
reference couplings.

\subsection{Admissible probability sections}
\label{sec:admissible-sections}

Let $Z_{\ge R}$ denote the
closure in $Z$ of the union of the layers from level $R$ onward.
Choose $R$ so that the geometric assumptions hold for every $n\ge R$,
and assume that, for some $L>0$,
\begin{equation}
 \diam Z_{\ge R}\le L.
 \label{eq:tail-diameter}
\end{equation}
We fix $R,L$ and write $W=W_{2,Z}$. 

Consider a segment $X_r,\ldots,X_{r+J}$ with $J\ge1$.
A \emph{probability section} over this segment is a sequence
$\boldsymbol\mu=(\mu_0,\ldots,\mu_J)$ with
$\mu_j\in\Pp(V_{r+j})$. Its action and endpoint distance are
\begin{equation}
 \mathcal A_{r,J}(\boldsymbol\mu)
 =J\sum_{j=0}^{J-1}W(\mu_j,\mu_{j+1})^2,
 \qquad D=W(\mu_0,\mu_J).
 \label{eq:section-action}
\end{equation}
The factor $J$ corresponds to the parameter increments $1/J$ on
$[0,1]$. The triangle inequality and Cauchy--Schwarz give
\begin{equation}
 D^2\le\left(\sum_{j=0}^{J-1}W(\mu_j,\mu_{j+1})\right)^2
 \le\mathcal A_{r,J}(\boldsymbol\mu).
 \label{eq:action-lower-bound}
\end{equation}

\begin{remark}[Joint transport formulation]
Writing $\Pi(\mu_0,\ldots,\mu_J)$ for the probability measures on
$V_r\times\cdots\times V_{r+J}$ with the prescribed marginals, we have
\[
 \mathcal A_{r,J}(\boldsymbol\mu)
 =\min_{\pi\in\Pi(\mu_0,\ldots,\mu_J)}
 J\int\sum_{j=0}^{J-1}d_Z(x_j,x_{j+1})^2\,\dd\pi.
\]
Every adjacent marginal has cost at least the corresponding squared
Wasserstein distance. Successive gluing of optimal adjacent couplings
attains their sum.
\end{remark}

\begin{definition}[Admissible probability sections]
\label{def:admissible-section}
For $\eta>0$, a probability section with endpoints $\mu_0,\mu_J$
is \emph{$\eta$-admissible} if
\begin{equation}
 \mathcal A_{r,J}(\boldsymbol\mu)\le D^2+\eta L^2.
 \label{eq:scd-action}
\end{equation}
The class of these sections is denoted by
$\mathcal G_{r,J}^{\eta}(\mu_0,\mu_J)$.

In particular, 
for $J=2$, a measure $\nu\in\Pp(V_{r+1})$ is an
\emph{$\eta$-admissible midpoint} of $\mu_-,\mu_+$ 
if
$(\mu_-,\nu,\mu_+)\in\mathcal G_{r,2}^{\eta}(\mu_-,\mu_+)$,
or equivalently if
\begin{equation}
  2W(\mu_-,\nu)^2+2W(\nu,\mu_+)^2 \le W(\mu_-,\mu_+)^2+\eta L^2.
  \label{eq:midpoint-action}
\end{equation}
\end{definition}

The admissible class depends only on the metric geometry and the
chosen tolerance. 
The next theorem shows that complete refinement
makes this class nonempty uniformly over the endpoints.

\begin{theorem}[Existence of admissible sections]
\label{thm:admissible-existence}
For any integer $J\ge2$
and  $\eta>0$, there exists
$r_{g}=r_{g}(J,\eta)\ge R$ such that,
for any $r\ge r_{g}$ and pair of probability measures
$\mu_0\in\Pp(V_r)$, $\mu_J\in\Pp(V_{r+J})$,
\[
 \mathcal G_{r,J}^{\eta}(\mu_0,\mu_J) \ne \varnothing.
\]
\end{theorem}

\begin{proof}
We first consider Dirac endpoint measures.
Let $x\in V_r$, $z\in V_{r+J}$, and put $c=d_Z(x,z)$.
The full support of the reference marginals allows us to follow
the support of the reference couplings backwards from $z$
to a point $z'\in V_r$.
Since $\ell_n$ is nonincreasing,
\eqref{eq:reference-displacement} gives
\[
 d_Z(z',z)\le CJ\ell_r.
\]

Choose a shortest path from $x$ to $z'$ in the graph of
local connections on $V_r$.
Its length satisfies
\[
 s=d_r(x,z')=d_Z(x,z')\le c+CJ\ell_r.
\]
Every connection on this path realizes the distance between its
endpoints; otherwise it could be replaced by a shorter path.
Its length is therefore at most $\ell_r$ by
\eqref{eq:local-resolution-scale}.

For $0\le j\le J$, choose a vertex $q_j$ nearest to the
arclength position $js/J$, with $q_0=x$ and $q_J=z'$.
The arclength error is at most $\ell_r/2$, and hence
\[
 d_r(q_j,q_{j+1})\le \frac{s}{J}+\ell_r.
\]
If $s=0$, take every $q_j=x$.

For $1\le j<J$, follow the supports of the reference couplings
for $j$ steps from $q_j$ to obtain $x_j\in V_{r+j}$.
Set $x_0=x$ and $x_J=z$.
The backward chain chosen above also gives the required estimate
at $j=J$, so
\[
 d_Z(q_j,x_j)\le Cj\ell_r
 \qquad(0\le j\le J).
\]
It follows that
\begin{align*}
 d_Z(x_j,x_{j+1})
 &\le Cj\ell_r+d_r(q_j,q_{j+1})+C(j+1)\ell_r\\
 &\le \frac{s}{J}+\ell_r+C(2j+1)\ell_r\\
 &\le \frac{c}{J}+O_{J,C}(\ell_r).
\end{align*}
Since $c\le L$, the resulting point section satisfies
\[
 J\sum_{j=0}^{J-1}d_Z(x_j,x_{j+1})^2
 \le c^2+O_{J,C}(L\ell_r+\ell_r^2).
\]
The error is uniform over the Dirac endpoint measures.
The complete refinement gives $\ell_r\to0$, 
we may therefore choose $r_{g}=r_{g}(J,\eta)\ge R$
such that, for every $r\ge r_{g}$ and every endpoint pair,
the constructed section satisfies
\[
 J\sum_{j=0}^{J-1}d_Z(x_j,x_{j+1})^2
 \le d_Z(x,z)^2+\eta L^2.
\]

Now let $\mu_0\in\Pp(V_r)$ and $\mu_J\in\Pp(V_{r+J})$,
and choose a $d_{Z}^{2}$-optimal coupling
$\pi\in\Pi(\mu_0,\mu_J)$.
For each pair $(x,z)$, select one of the point sections above
and define
\[
 \mu_j=\sum_{x,z}\pi(x,z)\,\delta_{x_j(x,z)}.
\]
The induced adjacent couplings give
\begin{align*}
 \mathcal A_{r,J}(\boldsymbol\mu)
 &\le
 \sum_{x,z}\pi(x,z)\,
 J\sum_{j=0}^{J-1}
 d_Z\bigl(x_j(x,z),x_{j+1}(x,z)\bigr)^2\\
 &\le
 \sum_{x,z}\pi(x,z)\bigl[d_Z(x,z)^2+\eta L^2\bigr]\\
 &=W(\mu_0,\mu_J)^2+\eta L^2.
\end{align*}
Thus $\boldsymbol\mu\in
\mathcal G_{r,J}^{\eta}(\mu_0,\mu_J)$,
with a starting level independent of the endpoint probabilities.
\end{proof}

The reference couplings in this proof select nearby representatives.
The adjacent couplings of the resulting section need not be supported
on the reference couplings. The theorem supplies an admissible class
at every prescribed accuracy; it makes no choice of a distinguished
section within that class.

\begin{lemma}[Distances to an admissible midpoint]
\label{lem:midpoint-defect}
If $\nu$ is an $\eta$-admissible midpoint of $\mu_-,\mu_+$ and
$D=W(\mu_-,\mu_+)$, then each of $W(\mu_-,\nu)$ and
$W(\nu,\mu_+)$ differs from $D/2$ by at most $L\sqrt\eta/2$.
\end{lemma}
\begin{proof}
Write $a=W(\mu_-,\nu)$ and $b=W(\nu,\mu_+)$. The triangle
inequality gives $b\ge|D-a|$, hence
\[
 D^2+4(a-D/2)^2=2a^2+2(D-a)^2
 \le2a^2+2b^2\le D^2+\eta L^2.
\]
Interchanging $a$ and $b$ gives the second bound.
\end{proof}

More generally, Lemma~\ref{lem:finite-section-distance} will show that
an $\eta$-admissible section satisfies
\[
 \left|W(\mu_i,\mu_k)-\frac{k-i}{J}D\right|
 \le L\sqrt\eta
 \qquad(0\le i<k\le J).
\]
Thus admissible sections approximate the distance relations of a
constant-speed Wasserstein geodesic at the parameters $j/J$.

\subsection{The SCD condition}

Fix $K\in\R$ and put $\kappa=K/N$. The comparison coefficients are
\[
 \mathfrak s_\kappa(a)=
 \begin{cases}
 \sin(\sqrt\kappa a)/\sqrt\kappa,&\kappa>0,\\
 a,&\kappa=0,\\
 \sinh(\sqrt{-\kappa}a)/\sqrt{-\kappa},&\kappa<0,
 \end{cases}
 \qquad
 \sigma_\kappa^{(t)}(a)=
 \frac{\mathfrak s_\kappa(ta)}{\mathfrak s_\kappa(a)},
 \quad \sigma_\kappa^{(t)}(0)=t.
\]
For $K>0$, we require
\begin{equation}
 KL^2<N\pi^2.
 \label{eq:strict-comparison-range}
\end{equation}
The coefficients are then nonnegative and smooth for
$0\le t\le1$ and $0\le a\le L$. At $t=1/2$, they reduce to
\begin{equation}
 S_\kappa(a):=\sigma_\kappa^{(1/2)}(a)=
 \begin{cases}
 [2\cos(\sqrt\kappa a/2)]^{-1},&\kappa>0,\\
 1/2,&\kappa=0,\\
 [2\cosh(\sqrt{-\kappa}a/2)]^{-1},&\kappa<0.
 \end{cases}
 \label{eq:mid-coefficient}
\end{equation}

\begin{definition}[The SCD condition]
\label{def:scd}
Under the standing assumptions above, 
a refinement system $\Sigma$ satisfies $\SCD(K,N)$
if, for every integer $J\ge2$ and every $\eps>0$, there exists
$r_0=r_0(J,\eps)\ge R$ such that, for every $r\ge r_0$ and every
pair $\mu_0\in\Pp(V_r)$, $\mu_J\in\Pp(V_{r+J})$, there is a section
\[
 \boldsymbol\mu\in\mathcal G_{r,J}^{\eps}(\mu_0,\mu_J)
\]
satisfying
\begin{equation}
 U_{r+j}(\mu_j)\ge
 \sigma_{K/N}^{(1-j/J)}(D)U_r(\mu_0)
 +\sigma_{K/N}^{(j/J)}(D)U_{r+J}(\mu_J)-\eps,
 \qquad 1\le j<J,
 \label{eq:scd-entropy}
\end{equation}
where $D=W(\mu_0,\mu_J)$.
\end{definition}

Theorem~\ref{thm:admissible-existence} ensures that admissible
sections exist on sufficiently late refinement segments
of each fixed length. The SCD condition requires a single
such section to satisfy \eqref{eq:scd-entropy}
at every interior index.
Moreover, 
for fixed $J$ and $\eps$, the starting level $r_0$
is uniform over the endpoint measures.
Using the same tolerance for admissibility and entropy
does not change the condition: separate tolerances can be met
by applying the definition with their minimum.

\begin{definition}[Midpoint condition]
\label{def:midpoint}
Under the same standing assumptions, 
a refinement $\Sigma$ satisfies
$\SCDmid(K,N)$ if, for every $\eps>0$, there exists
$r_0=r_0(\eps)\ge R$ such that, for every $r\ge r_0$ and every
pair $\mu_-\in\Pp(V_r)$, $\mu_+\in\Pp(V_{r+2})$, there is an
$\eps$-admissible midpoint $\nu\in\Pp(V_{r+1})$ satisfying
\begin{equation}
 U_{r+1}(\nu)\ge S_{K/N}(D)\bigl[U_r(\mu_-)+U_{r+2}(\mu_+)\bigr]-\eps,
 \label{eq:midpoint-entropy}
\end{equation}
where $D=W(\mu_-,\mu_+)$.
\end{definition}

The midpoint condition compares entropy at the same measure whose
distances are controlled by admissibility. 
It is the $J=2$ case of $\SCD(K,N)$. 
Perhaps surprisingly, this special case already implies
the full condition.
We prove their equivalence in the next section.

\section{Midpoint characterization}
\label{sec:direct-midpoint}

In this section, we prove that $\SCD(K,N)$ and $\SCDmid(K,N)$ are equivalent.
The main point is to obtain all entropy comparisons on the same
admissible section.
Starting with a section supplied by
Theorem~\ref{thm:admissible-existence}, we successively replace
its interior measures using the midpoint condition 
and show that finitely many replacements suffice to obtain
the required entropy comparisons while keeping the action
within the admissible bound.

Throughout this section, we retain the notation and assumptions
of Section~\ref{sec:entropy} 
and write $C_S=\Lip(S_\kappa|_{[0,L]})$
for the Lipschitz constant of the midpoint coefficient.

\begin{theorem}[Midpoint characterization]
\label{thm:direct-midpoint}
Under the standing assumptions, 
a refinement system $\Sigma$ satisfies $\SCD(K,N)$
if and only if it satisfies $\SCDmid(K,N)$.
\end{theorem}

The implication $\SCD(K,N)\Rightarrow\SCDmid(K,N)$
follows by taking $J=2$.
For the converse, 
we first establish a few preliminary lemmas
to estimate distances along
admissible sections and control the effect of
a single midpoint replacement.

\subsection{Action stability and midpoint replacement}

\begin{lemma}[Distances along an admissible section]
\label{lem:finite-section-distance}
Let $\boldsymbol\mu\in\mathcal G_{r,J}^{\eta}(\mu_0,\mu_J)$
and $D=W(\mu_0,\mu_J)$. Then
\begin{equation}
 \left|W(\mu_i,\mu_k)-\frac{k-i}{J}D\right|
 \le L\sqrt\eta,
 \qquad 0\le i<k\le J.
 \label{eq:finite-section-distance}
\end{equation}
\end{lemma}

\begin{proof}
Write $a_j=W(\mu_j,\mu_{j+1})$ and
$T=\sum_{j=0}^{J-1}a_j\ge D$.
Admissibility gives
\[
 J\sum_{j=0}^{J-1}(a_j-D/J)^2
 =\mathcal A_{r,J}(\boldsymbol\mu)-2DT+D^2
 \le\eta L^2.
\]
Hence, for every $I\subset\{0,\ldots,J-1\}$,
Cauchy--Schwarz yields
\[
 \left|\sum_{j\in I}(a_j-D/J)\right|
 \le L\sqrt{\frac{\eta|I|}{J}}.
\]

Fix $i<k$, set $m=k-i$ and $I=\{i,\ldots,k-1\}$.
The triangle inequality gives
\begin{align*}
 W(\mu_i,\mu_k)
 &\le\sum_{j\in I}a_j
 \le\frac{mD}{J}+L\sqrt{\frac{\eta m}{J}},\\
 W(\mu_i,\mu_k)
 &\ge D-\sum_{j\notin I}a_j
 \ge\frac{mD}{J}-L\sqrt{\frac{\eta(J-m)}{J}}.
\end{align*}
Both error terms are at most $L\sqrt\eta$.
\end{proof}

\begin{lemma}[Action and entropy under midpoint replacement]
\label{lem:midpoint-replacement}
Let $\boldsymbol\mu=(\mu_0,\ldots,\mu_J)$ be a probability
section over $X_r,\ldots,X_{r+J}$.
Fix $\delta>0$ and an index $1\le j<J$.
Let $\nu\in\Pp(V_{r+j})$ be a $\delta$-admissible midpoint
of $\mu_{j-1},\mu_{j+1}$ satisfying the entropy comparison
in Definition~\ref{def:midpoint} with tolerance $\delta$.
Denote by $\boldsymbol\mu^{\rm new}$ the section obtained
by replacing $\mu_j$ with $\nu$, and put $D=W(\mu_0,\mu_J)$.
Then
\[
 \mathcal A_{r,J}(\boldsymbol\mu^{\rm new})
 \le\mathcal A_{r,J}(\boldsymbol\mu)
       +\frac{J\delta L^2}{2}.
\]
If, moreover,
$\boldsymbol\mu\in\mathcal G_{r,J}^{\bar\eta}(\mu_0,\mu_J)$
for some $\bar\eta>0$, then
\begin{equation}
 U_{r+j}(\nu)
 \ge s\bigl[U_{r+j-1}(\mu_{j-1})
            +U_{r+j+1}(\mu_{j+1})\bigr]-e,
 \label{eq:finite-entropy-update}
\end{equation}
where $s=S_\kappa(2D/J)$ and
$e=\delta+2C_SL\sqrt{\bar\eta}$.
\end{lemma}

\begin{proof}
Set
\[
 a=W(\mu_{j-1},\mu_j),\qquad
 b=W(\mu_j,\mu_{j+1}),\qquad
 c=W(\mu_{j-1},\mu_{j+1}).
\]
The triangle inequality gives
$a^2+b^2\ge(a+b)^2/2\ge c^2/2$,
whereas the $\delta$-admissibility of $\nu$ gives
\[
 W(\mu_{j-1},\nu)^2+W(\nu,\mu_{j+1})^2
 \le\frac{c^2+\delta L^2}{2}.
\]
Only these two contributions to the action change.
Subtracting the old contribution and multiplying by $J$
proves the action estimate.

Now suppose that
$\boldsymbol\mu\in\mathcal G_{r,J}^{\bar\eta}(\mu_0,\mu_J)$.
Lemma~\ref{lem:finite-section-distance} gives
\[
 |c-2D/J|\le L\sqrt{\bar\eta}.
\]
Both $c$ and $2D/J$ lie in $[0,L]$, so
\[
 |S_\kappa(c)-s|\le C_SL\sqrt{\bar\eta}.
\]
Write $u_{j-1}=U_{r+j-1}(\mu_{j-1})$ and
$u_{j+1}=U_{r+j+1}(\mu_{j+1})$.
The entropy comparison for $\nu$ and the bound $0<U_n\le1$
yield
\begin{align*}
 U_{r+j}(\nu)
 &\ge S_\kappa(c)(u_{j-1}+u_{j+1})-\delta\\
 &\ge s(u_{j-1}+u_{j+1})
       -2C_SL\sqrt{\bar\eta}-\delta.
\end{align*}
This proves \eqref{eq:finite-entropy-update}.
\end{proof}

\subsection{Simultaneous entropy comparisons}

Fix the endpoint measures and put
$D=W(\mu_0,\mu_J)$, $u_0=U_r(\mu_0)$, and
$u_J=U_{r+J}(\mu_J)$.
Define the comparison values
\begin{equation}
 h_j=\sigma_\kappa^{(1-j/J)}(D)u_0
       +\sigma_\kappa^{(j/J)}(D)u_J,
 \qquad 0\le j\le J.
 \label{eq:finite-entropy-target}
\end{equation}
The addition formulas for sine and hyperbolic sine give
\begin{equation}
 h_j=s(h_{j-1}+h_{j+1}),\qquad
 s=S_\kappa(2D/J),\qquad 1\le j<J.
 \label{eq:finite-target-recurrence}
\end{equation}
When $\kappa=0$ or $D=0$, the sequence $h_j$ is affine,
and the same identity holds with $s=1/2$.

Since $0<u_0,u_J\le1$, the comparison values satisfy
\begin{equation}
 0\le h_j\le C_h,\qquad
 C_h=
 \begin{cases}
 1,&\kappa\le0,\\
 1/\cos(\sqrt\kappa L/2),&\kappa>0.
 \end{cases}
 \label{eq:finite-target-bound}
\end{equation}
Indeed, for $\kappa\le0$, convexity of $\mathfrak s_\kappa$
gives $\sigma_\kappa^{(t)}(D)\le t$ for $0\le t\le1$.
For $\kappa>0$, the addition formula gives
\[
 \sigma_\kappa^{(1-t)}(D)+\sigma_\kappa^{(t)}(D)
 =\frac{\cos((t-1/2)\sqrt\kappa D)}
        {\cos(\sqrt\kappa D/2)}
 \le\frac{1}{\cos(\sqrt\kappa L/2)}.
\]
The identity at $D=0$ is understood by continuity.

The next lemma turns the local estimate
\eqref{eq:finite-entropy-update} into simultaneous lower bounds
at all interior indices.
\begin{lemma}[Contraction of the entropy deficit]
\label{lem:finite-entropy-contraction}
Fix $J\ge3$ and set
\[
 w_j=\sin(\pi j/J),\qquad
 w_* =\min_{1\le j<J}w_j>0,
 \qquad
 q_J=
 \begin{cases}
 \cos(\pi/J),&\kappa\le0,\\
 \displaystyle\frac{\cos(\pi/J)}{\cos(\sqrt\kappa L/J)},
     &\kappa>0.
 \end{cases}
\]
Then $0<q_J<1$. 
Let $e\ge0$ and keep $u_0=h_0$, $u_J=h_J$ fixed. 
A \emph{sweep} consists of replacing
the interior values in the order $1,\ldots,J-1$, 
requiring
$u_j^{\rm new}\ge s(u_{j-1}+u_{j+1})-e$ 
at each replacement, 
with the current neighboring values.
If the initial values are nonnegative, 
then after $B$ times of sweep,
\begin{equation}
 \max_{1\le j<J}\frac{(h_j-u_j)_+}{w_j}
 \le q_J^B\frac{C_h}{w_*}
       +\frac{e}{(1-q_J)w_*}.
 \label{eq:finite-sweep-contraction}
\end{equation}
\end{lemma}

\begin{proof}
For $\kappa\le0$, we have $0<q_J=\cos(\pi/J)<1$.
For $\kappa>0$, the strict comparison range gives
$0\le\sqrt\kappa L/J<\pi/J\le\pi/3$
and hence $0<q_J<1$.
The sine weights satisfy
\[
 w_{j-1}+w_{j+1}=2\cos(\pi/J)w_j.
\]
Using the formula for $S_\kappa$ and $D\le L$, we obtain
\[
 s(w_{j-1}+w_{j+1})
 =2s\cos(\pi/J)w_j
 \le q_Jw_j.
\]

Set $d_j=(h_j-u_j)_+ := \max\{h_{j} - u_{j},0\}$. 
Since the endpoints are fixed, $d_0=d_J=0$.
For each replacement, write
$d_j^{\rm new} = (h_{j} - u_{j}^{\rm new})_+$.
The recurrence \eqref{eq:finite-target-recurrence}
and the replacement inequality imply
\[
 d_j^{\rm new}\le s(d_{j-1}+d_{j+1})+e.
\]
Put $F=e/((1-q_J)w_*)$ and suppose that at the start of a sweep
$d_j\le(F+E)w_j$ for every interior index, where $E \ge 0$.
Whenever the current neighboring deficits satisfy this bound, 
the next replacement gives
\begin{align*}
 d_j^{\rm new}
 &\le s(F+E)(w_{j-1}+w_{j+1})+e\\
 &\le q_J(F+E)w_j+e\\
 &\le (F+q_JE)w_j,
\end{align*}
where we use $e=(1-q_J)Fw_*\le(1-q_J)Fw_j$.
Since $F+q_JE\le F+E$, 
the original bound remains valid throughout the sweep.
Each interior value is replaced once, so at the end
all deficits satisfy the improved bound 
$d_{j} \le (F+q_{J}E)w_{j}$.

Initially $u_j \ge 0$ and $h_j \le C_h$, 
so $d_j/w_j \le C_h/w_*$. 
Starting with $E_0=C_h/w_*$ 
and iterating the sweep estimate 
therefore gives
\[
 \frac{d_j}{w_j}
  \le F+q_J^BE_0
  =  \frac{e}{(1-q_J)w_*} + q_J^B\frac{C_h}{w_*}.
\]
This proves \eqref{eq:finite-sweep-contraction}.
\end{proof}

\begin{proof}[Proof of Theorem~\ref{thm:direct-midpoint}]
The implication $\SCD\Rightarrow\SCDmid$ 
is the $J=2$ case of the definition.
Assume $\SCDmid(K,N)$ 
and fix $J \ge 3$ and $\eps > 0$.
We first choose the number of sweeps, 
then the action tolerance,
and finally the accuracy of each midpoint replacement.

Since $0<q_J<1$, choose an integer $B\ge1$ such that
\[
 q_J^B\frac{C_h}{w_*}\le\frac{\eps}{2},
\]
and put $M=B(J-1)$ as the total number of midpoint replacements.
Choose $0<\bar\eta\le\eps$ so that
\[
 2C_SL\sqrt{\bar\eta}
 \le\frac14(1-q_J)w_*\eps.
\]
Finally, choose $\delta>0$ satisfying
\begin{equation}
 \frac{MJ\delta}{2}\le\frac{\bar\eta}{2},
 \qquad
 \delta\le\frac14(1-q_J)w_*\eps.
 \label{eq:finite-parameter-order}
\end{equation}

By Theorem~\ref{thm:admissible-existence}, 
there exists  $r_0=r_0(J,\eps)\ge R$ such that
\[
 \mathcal G_{r,J}^{\bar\eta/2}(\mu_0,\mu_J)\ne\varnothing
\]
for every $r\ge r_0$ and every pair
$\mu_0\in\Pp(V_r)$, $\mu_J\in\Pp(V_{r+J})$.
By $\SCDmid(K,N)$, we may increase $r_0$ so that,
for every $q\ge r_0$, each pair of endpoint measures
on $V_q,V_{q+2}$ admits a $\delta$-admissible midpoint
satisfying the entropy comparison in
Definition~\ref{def:midpoint} with tolerance $\delta$.

Fix $r\ge r_0$ and endpoints
$\mu_0\in\Pp(V_r)$, $\mu_J\in\Pp(V_{r+J})$.
Put $D=W(\mu_0,\mu_J)$ and let $h_j$ be the comparison
values in \eqref{eq:finite-entropy-target}.
Choose an initial section
\[
 \boldsymbol\mu^{(0)}
 \in\mathcal G_{r,J}^{\bar\eta/2}(\mu_0,\mu_J).
\]
Perform $B$ times of sweeps.
The choice of $r_0$ ensures that at each step of replacements,
there exists a $\delta$-admissible midpoint
of the current neighboring measures 
that satisfies the midpoint
entropy comparison with tolerance $\delta$.

Write $\boldsymbol\mu^{(k)}=(\mu_0^{(k)},\ldots,\mu_J^{(k)})$
for the section after $k$ individual replacements.
The action estimate in Lemma~\ref{lem:midpoint-replacement}
gives
\begin{align*}
\mathcal A_{r,J}(\boldsymbol\mu^{(k)})
  \le\mathcal A_{r,J}(\boldsymbol\mu^{(0)}) + \frac{kJ\delta L^2}{2}
  \le D^2+ \left(\frac{\bar\eta}{2}+\frac{MJ\delta}{2}\right)L^2
  \le D^2+\bar\eta L^2.
\end{align*}
Thus every intermediate section is $\bar\eta$-admissible.

The endpoints remain fixed, so $s=S_\kappa(2D/J)$
is unchanged throughout the sweeps.
The entropy estimate in Lemma~\ref{lem:midpoint-replacement}
therefore applies at every step with the same error bound
\[
 e=\delta+2C_SL\sqrt{\bar\eta}
 \le\frac12(1-q_J)w_*\eps.
\]
Since the initial transformed entropies are positive,
Lemma~\ref{lem:finite-entropy-contraction} yields
\[
 \max_{1\le j<J}
 \frac{\bigl(h_j-U_{r+j}(\mu_j^{(M)})\bigr)_+}{w_j}
 \le q_J^B\frac{C_h}{w_*}
       +\frac{e}{(1-q_J)w_*}
 \le\eps.
\]
As $w_j\le1$, it follows that
\[
 U_{r+j}(\mu_j^{(M)})\ge h_j-\eps,
 \qquad 1\le j<J.
\]
Moreover, $\bar\eta\le\eps$, so the same final section belongs to
$\mathcal G_{r,J}^{\eps}(\mu_0,\mu_J)$.
It therefore satisfies all the requirements of
Definition~\ref{def:scd}.
The starting level $r_0$ is independent of the endpoint measures,
which proves $\SCD(K,N)$.
\end{proof}

\begin{remark}[Scope of the equivalence]
\label{rem:direct-scope}
The argument applies to each fixed segment length $J$.
The starting level may depend on $J$ and the tolerance,
but is uniform over the endpoint measures.

The refinement geometry enters the proof through
Theorem~\ref{thm:admissible-existence}.
Under the same metric and entropy setup, with the stated
tail diameter bound and coefficient range,
the argument therefore applies whenever, for each $J\ge2$
and every $\eta>0$, the classes
$\mathcal G_{r,J}^{\eta}(\mu_0,\mu_J)$
are nonempty for all sufficiently large $r$,
uniformly over the endpoint measures.
\end{remark}

\section{Measured Gromov--Hausdorff convergence}
\label{sec:limits}

We now study measured Gromov--Hausdorff convergence
in the intrinsic bridge realization.
We show that any compact limit of a complete refinement
is a geodesic space, and give sufficient conditions for convergence.
These results provide the geometric setting
for the limit characterization of $\SCD(K,N)$ in the next section.

\subsection{Convergence and limit geometry}

For compact metric probability spaces, 
we use the measured
Gromov--Hausdorff topology induced by the
Gromov--Hausdorff--Prokhorov distance
\[
 d_{\mathrm{GHP}}\bigl((A,d_A,\alpha),(B,d_B,\beta)\bigr)
 =
 \inf_{E,\iota,\jmath}
 \max\left\{
 d_H^E\bigl(\iota(A),\jmath(B)\bigr),
 d_P^E\bigl(\iota_\#\alpha,\jmath_\#\beta\bigr)
 \right\}.
\]
Here the infimum is over isometric embeddings into a common
metric space $E$, and $d_H$ and $d_P$ denote the Hausdorff
and Prokhorov distances; see \cite[Section~2.2]{ADH2013}.

In our setting, the bridges already determine a common realization
$(Z,d_Z)$ by Lemma~\ref{lem:gluing}.
We formulate convergence in this space
so that the limit geometry and the SCD comparisons
use the same metric.

\begin{definition}[Convergence in the intrinsic bridge realization]
\label{def:convergence}
A comparison system $\Sigma$ has measured Gromov--Hausdorff
convergence in its intrinsic bridge realization if there exist
a nonempty compact set $M \subset Z$ 
and a probability measure $m$ on $Z$ such that
\begin{equation}
 a_{n} := d_{H}^{Z}(V_{n}, M) \longrightarrow 0,
 \qquad
 m_{n} \rightharpoonup m,
 \qquad
 \supp m = M.
 \label{eq:convergence}
\end{equation}
Weak convergence is taken in $Z$.
We write $d = d_Z |_{M\times M}$ and call $(M, d, m)$ the limit.
\end{definition}

\begin{remark}
Definition \ref{def:convergence} fixes the realization in which convergence occurs.
Abstract measured Gromov--Hausdorff convergence alone
does not ensure compatibility with the prescribed bridges.
\end{remark}

\begin{lemma}[Compactness of the intrinsic bridge realization]
\label{lem:compact}
Under \eqref{eq:convergence}, the space $Z$ is compact.
Moreover,
\[
 d_{\mathrm{GH}}\bigl((V_n,d_n),(M,d)\bigr)
 \le a_n\longrightarrow0,
 \qquad
 d_{\mathrm{GHP}}\bigl(X_n,(M,d,m)\bigr)\longrightarrow0.
\]
\end{lemma}

\begin{proof}
Fix $\eta>0$ and choose a finite $\eta/3$-net of $M$.
For all sufficiently large $n$, every point of $V_n$
lies within $\eta/3$ of $M$, and hence within $2\eta/3$
of this net.
The earlier layers form a finite set.
Adding their points gives a finite $\eta$-net
for the union of all layers.
This union is therefore totally bounded,
and its completion $Z$ is compact.

The isometric inclusions into $Z$ give
\[
 d_{\mathrm{GH}}\bigl((V_n,d_n),(M,d)\bigr)\le a_n.
\]
On compact $Z$, weak convergence of probability measures
is equivalent to convergence in Prokhorov distance.
Therefore we have
\[
 d_{\mathrm{GHP}}\bigl(X_n,(M,d,m)\bigr)
 \le \max\{a_n,d_P^Z(m_n,m)\}
 \longrightarrow 0.
\]
\end{proof}

For a complete refinement, the local geometric scales
also determine the geometry of the limit.
The shortest-path metrics provide approximate midpoints
whose errors vanish with $\ell_n$.

\begin{lemma}[Geodesicity of the limit]
\label{lem:geodesic-limit}
Let $\Sigma$ be a complete refinement system as in
Section~\ref{sec:refinement}, 
and suppose that \eqref{eq:convergence} holds.
Then $(M, d)$ is a geodesic space.
\end{lemma}

\begin{proof}
Fix $x, y\in M$ and choose $x_n, y_n\in V_n$
such that $x_n \to x$ and $y_n \to y$ in $Z$.
Join $x_n$ to $y_n$ by a shortest path
in the local connection of $V_n$.

Every subpath is shortest.
In particular, the length of each connection on this path
equals the $d_n$-distance between its endpoints,
and is therefore at most $\ell_n$
by \eqref{eq:local-resolution-scale}.
Choose a vertex $z_n$ closest to the arc length midpoint.
If $x_n = y_n$, take $z_n = x_n$.
Then
\[
 \left|d_n(x_n,z_n)-\frac12d_n(x_n,y_n)\right|
 \le\frac{\ell_n}{2},
 \qquad
 \left|d_n(z_n,y_n)-\frac12d_n(x_n,y_n)\right|
 \le\frac{\ell_n}{2}.
\]

By Lemma~\ref{lem:compact}, a subsequence of $(z_n)$
converges in $Z$ to some $z$.
Since $d_Z(z_n, M) \le a_n \to 0$, we have $z \in M$.
The layers embed isometrically in $Z$, and $\ell_n \to 0$,
so passage to the limit gives
\[
 d(x, z) = d(z, y) = \frac12 d(x, y).
\]
Thus every pair of points in $M$ has a midpoint.

Choose midpoints recursively to define $\gamma$
at the dyadic points of $[0,1]$, 
with
$\gamma(0)=x$ and $\gamma(1)=y$.
At each depth, the sum of the consecutive distances
equals $d(x, y)$.
The triangle inequality therefore gives
\[
 d(\gamma(s),\gamma(t))=|s-t|d(x,y)
\]
for all dyadic $s$ and $t$.
Completeness of $M$ extends $\gamma$ to $[0,1]$
with the same distance identity.
Hence $\gamma$ is a constant-speed geodesic.
\end{proof}

The geodesicity conclusion follows 
from complete refinement and Hausdorff convergence.
Furthermore, Hausdorff convergence gives $M \subset Z_{\ge R}$.
Thus $\diam M\le L$, with $R,L$ and the comparison range
retained from Section~\ref{sec:entropy}.

\subsection{Endpoint control and compact measured limits}

A reference path over the refinement segment $X_n,\ldots,X_r$
is a sequence $(x_n,\ldots,x_r)$ with $x_k\in V_k$ such that
for any $n \le k< r$, we have
\[
 \gamma_k(x_k,x_{k+1})>0.
\]
For sufficiently large $n$, the displacement estimate
\eqref{eq:reference-displacement} gives
\[
 d_Z(x_n,x_r)\le C\sum_{k=n}^{r-1}\ell_k
\]
along every supported reference path.
For each fixed $r-n$, the right-hand side tends to zero.
To obtain a sufficient condition for convergence,
we now control the endpoint distance uniformly over all $r>n$.
\begin{definition}[Endpoint control]
\label{def:controlled}
A comparison system is endpoint-controlled if there are
nonnegative numbers $\omega_{n}\rightarrow 0$ such that
\[
 d_Z(x_n,x_r)\le \omega_{n}
\]
for every reference path $x_n, \ldots, x_r$.
\end{definition}

This estimate is uniform over the initial element,
the terminal layer, and the intermediate elements.
Moreover, it is equivalent to the same
estimate with $D_{[n,r]}$ in place of $d_Z$
by Lemma~\ref{lem:gluing}.
The stronger path-length bound
\begin{equation}
 \sum_{k=n}^{r-1}b_k(x_k,x_{k+1})\le \omega_{n}
 \label{eq:path-length-sufficient}
\end{equation}
implies endpoint control.

For both the endpoint and path-length estimates,
it suffices to consider all sufficiently large starting levels.
Indeed, suppose the estimate holds from level $n_*$ onward, set
\[
 B_k=\max_{\gamma_k(x,y)>0}b_k(x,y),
 \qquad
 \omega_{n}=\omega_{n_*}+\sum_{k=n}^{n_*-1}B_k
 \quad(n<n_*).
\]
Each $B_k$ is finite.
The sum bounds the part of a reference path
before level $n_*$.
Thus the extended sequence gives the required estimate
at every starting layer.
Moreover, $\omega_{n}-\omega_{n+1}=B_n\ge0$ for $n<n_*$,
and the sequence tends to zero.

Endpoint control yields a compact measured limit
and couplings to the limit with uniformly bounded displacement.

\begin{theorem}[Compact limit from endpoint control]
\label{thm:compact-limit}
Let $\Sigma$ be an endpoint-controlled comparison system.
Its intrinsic bridge realization $(Z,d_Z)$ is compact.
There exist a nonempty compact set $M\subset Z$
and a probability measure $m$ with $\supp m=M$ such that
\begin{equation}
 d_H^Z(V_n,M)\le \omega_{n},
 \qquad
 W_{\infty,Z}(m_n,m)\le \omega_{n}.
 \label{eq:controlled-limit}
\end{equation}
In particular, \eqref{eq:convergence} holds.
For each $n$, there is a coupling
$\beta_n\in\Pi(m_n,m)$ supported on
\[
 \{(x,z)\in V_n\times M:d_Z(x,z)\le \omega_{n}\}.
\]
Furthermore, for every integer $R\ge0$,
\begin{equation}
 \diam Z_{\ge R}\le\diam V_R+2\omega_{R}.
 \label{eq:diameter-bound}
\end{equation}
\end{theorem}

\begin{proof}
We first construct the limiting compact set,
then obtain its reference measure from the reference paths.

The positive reference masses and marginal identities ensure
that every element of $V_n$ can be continued along a reference
path to any later layer.
Conversely, every element of a later layer can be reached
by a reference path starting in $V_n$.
Endpoint control therefore gives
\[
 d_H^Z(V_n,V_r)\le \omega_{n}
 \qquad(r>n).
\]

For any $\eta>0$, choose $n$ with $\omega_{n}<\eta$.
The finite set $V_n$ covers every later layer
by balls of radius $\eta$.
Adding the finitely many earlier layers gives a finite
$\eta$-net for the union of all layers.
This union is totally bounded,
so its completion $Z$ is compact.

The displayed Hausdorff estimate also shows that $\{ V_n \}_{n}$
is Cauchy in Hausdorff distance.
It therefore converges to a nonempty compact set $M\subset Z$.
Letting $r\to\infty$ gives
\[
 d_H^Z(V_n,M)\le \omega_{n}.
\]

To construct the reference measure,
consider the Markov chain $\{ Y_n \}$ with initial law $m_0$
and transition kernels $P_n$ from
\eqref{eq:reference-kernels}.
The marginal identities give
$\operatorname{Law}(Y_n)=m_n$ for every $n$.
Almost surely, every transition has positive reference mass.
Hence endpoint control gives, simultaneously for all $r>n$,
\[
 d_Z(Y_n,Y_r)\le \omega_{n}.
\]
Thus $\{ Y_n \}$ is almost surely Cauchy and converges in $Z$
to a random point $Y_\infty$.
Since
\[
 d_Z(Y_n,M)\le \omega_{n}\longrightarrow0,
\]
the limit belongs to $M$ almost surely.
Passing to the limit in the endpoint estimate also gives
\[
 d_Z(Y_n,Y_\infty)\le \omega_{n}
 \qquad\text{almost surely}.
\]

Define
\[
 m=\operatorname{Law}(Y_\infty),
 \qquad
 \beta_n=\operatorname{Law}(Y_n,Y_\infty).
\]
Then $\beta_n\in\Pi(m_n,m)$ has the asserted support,
which proves the $W_\infty$ estimate.
Almost sure convergence yields weak convergence 
$m_n\rightharpoonup m$.

It remains to prove that $m$ has full support on $M$.
Fix $z\in M$ and $a>0$.
Choose $n$ with $\omega_{n}<a/4$ and $x\in V_n$ such that
$d_Z(x,z)\le \omega_{n}$.
The event $\{Y_n=x\}$ has probability $m_n(x)>0$.
On this event,
\[
 d_Z(Y_\infty,z)
 \le d_Z(Y_\infty,x)+d_Z(x,z)
 \le2\omega_{n}<a
\]
almost surely.
Consequently,
\[
 m(B_a(z))\ge m_n(x)>0.
\]
Since $m$ is concentrated on $M$, 
we have $\supp m=M$.

Finally, fix $R\ge0$.
Every element of a layer with index at least $R$
can be joined to an element of $V_R$
by a reference path.
Its distance from that element is at most $\omega_{R}$.
Comparing two such elements through their corresponding
points in $V_R$ gives
\[
 \diam\left(\bigcup_{n\ge R}V_n\right)
 \le\diam V_R+2\omega_{R}.
\]
Taking the closure proves \eqref{eq:diameter-bound}.
\end{proof}

\begin{remark}
If $\Sigma$ is a complete refinement,
Lemma~\ref{lem:geodesic-limit} further shows that
the resulting limit $(M,d,m)$ is geodesic.
\end{remark}

\subsection{Local criteria for endpoint control}

We now give local criteria for the endpoint control
used in Theorem~\ref{thm:compact-limit}.
These criteria either bound the accumulated displacement
along reference paths or confine them near retained anchors.
We first combine the displacement estimate of
Section~\ref{sec:refinement} with uniform scale contraction.

\begin{proposition}[Uniform scale contraction]
\label{prop:local-control}
Suppose a refinement system satisfies
\eqref{eq:reference-displacement} and, for some $0<q<1$,
\[
 \ell_{n+1}\le q\ell_n
\]
for all sufficiently large $n$.
Then the system is endpoint-controlled, with
\[
 \omega_{n}=\frac{C\ell_n}{1-q}.
\]
\end{proposition}

\begin{proof}
For every sufficiently large $n$, a supported reference path
satisfies
\[
 \sum_{k=n}^{r-1}b_k(x_k,x_{k+1})
 \le C\sum_{k=n}^{r-1}\ell_k
 \le C\ell_n\sum_{j=0}^{\infty}q^j
 =\frac{C\ell_n}{1-q}.
\]
This bound decreases to zero.
The preceding extension includes the finitely many earlier layers.
\end{proof}

Uniform contraction of the maximal scale is not needed
when each displacement can be bounded by the corresponding
decrease in local geometric scale.

\begin{proposition}[Local displacement budget]
\label{prop:local-budget}
Suppose a complete refinement satisfies, for some $C_0>0$
and all sufficiently large $n$,
\begin{equation}
 b_n(x,y)\le C_0\bigl(s_n(x)-s_{n+1}(y)\bigr)
 \quad\text{whenever }\gamma_n(x,y)>0.
 \label{eq:local-budget}
\end{equation}
Then \eqref{eq:path-length-sufficient} holds
with $\omega_{n}=C_0\ell_n$, and the system is endpoint-controlled.
\end{proposition}

\begin{proof}
Summing along a supported reference path gives
\[
 \sum_{k=n}^{r-1}b_k(x_k,x_{k+1})
 \le C_0\bigl(s_n(x_n)-s_r(x_r)\bigr)
 \le C_0\ell_n.
\]
Complete refinement gives $\ell_n \rightarrow 0$.
Extending the bound across the earlier layers proves the claim.
\end{proof}

The budget must hold for every pair exchanging positive
reference mass.
It is stronger than the resolution condition
\eqref{eq:local-resolution}, which compares the new scale
with the maximum over its geometric dependency set.

For example, consider uniform interval bisection.
Take interval lengths as element sizes,
midpoint distances as local connection lengths,
and Euclidean midpoint distances as bridges.
Transfer the reference mass of each interval to its two halves.
Then
\[
 s_n(x)=\ell_n,\qquad
 s_{n+1}(y)=\frac{\ell_n}{2},\qquad
 b_n(x,y)=\frac{\ell_n}{4}
\]
on every supported transition.
Thus \eqref{eq:local-budget} holds with $C_0=1/2$.

The same budget also applies to adaptive refinement,
where some reference transitions leave the geometry unchanged.

\begin{corollary}[Active and inactive reference transitions]
\label{cor:adaptive-control}
Suppose every supported reference transition is either
inactive, with
\[
 b_n(x,y)=0,\qquad s_{n+1}(y)=s_n(x),
\]
or active, with fixed constants $0<q<1$ and $A>0$ such that
\[
 s_{n+1}(y)\le q s_n(x),\qquad
 b_n(x,y)\le A s_n(x).
\]
If every infinite supported reference path has infinitely many
active transitions, then $\ell_n\to0$ and the system is
endpoint-controlled with
\[
 \omega_{n}=\frac{A\ell_n}{1-q}.
\]
\end{corollary}

\begin{proof}
Let $k_n$ be the minimum number of active transitions
in a supported reference path from layer $0$ to layer $n$.
We claim that $k_n\to\infty$.
Otherwise, some integer $B$ would bound this number
along paths of arbitrarily large length.

The tree of supported paths with at most $B$ active transitions
is finitely branching and has arbitrary depth.
Since $V_0$ is finite, one initial element has extensions
of arbitrary length.
Repeatedly choosing a successor with this property
produces an infinite supported path with at most $B$
active transitions, a contradiction.

The positive reference masses and marginal identities ensure
that every layer element is reached by a supported path
starting in $V_0$.
Consequently,
\[
 \ell_n\le\ell_0q^{k_n}\longrightarrow0.
\]
For an active transition,
\[
 b_n(x,y)
 \le\frac{A}{1-q}
       \bigl(s_n(x)-s_{n+1}(y)\bigr).
\]
For an inactive transition, both sides vanish.
Proposition~\ref{prop:local-budget} therefore applies.
\end{proof}

The preceding criteria control the total length of reference paths.
We now obtain endpoint control by keeping later elements
near retained anchors.
Successive anchors have zero bridge distance
and therefore represent the same point of $Z$.

\begin{proposition}[Confinement by retained anchors]
\label{prop:anchor-control}
Let $\Sigma$ be a complete refinement system.
Suppose there is a constant $C > 0$ with the following property.
Every $x\in V_n$ carries a mark $\tau=\tau(n,x)$
with radius $R_\tau$ and anchor $a_n^\tau\in V_n$ such that
\[
 d_n(x, a_n^\tau)\le R_\tau\le C\ell_n.
\]
For every subsequent refinement step $k\ge n$:
\begin{enumerate}
\item the anchor has a successor $a_{k+1}^\tau\in V_{k+1}$
      satisfying
      \[
       b_k(a_k^\tau,a_{k+1}^\tau)=0;
      \]
\item whenever $u\in V_k$ carries $\tau$ and
      $\gamma_k(u,v)>0$, the element $v\in V_{k+1}$
      inherits $\tau$ and satisfies
      \[
       d_{k+1}(v,a_{k+1}^\tau)\le R_\tau.
      \]
\end{enumerate}
Then $\Sigma$ is endpoint-controlled with $\omega_{n}=2C\ell_n$.
\end{proposition}

\begin{proof}
Fix a reference path $x_n, \ldots, x_r$
and let $\tau=\tau(n, x_n)$.
The mark is inherited along the path, thus
\[
 d_r(x_r,a_r^\tau)\le R_\tau.
\]
By Lemma~\ref{lem:gluing},
the zero bridge distances between successive anchors give
\[
 d_Z(a_n^\tau,a_r^\tau)=0.
\]
The triangle inequality therefore yields
\[
 \begin{aligned}
 d_Z(x_n,x_r)
 \le d_n(x_n,a_n^\tau)+d_r(a_r^\tau,x_r)
 \le 2R_\tau
 \le 2C\ell_n.
 \end{aligned}
\]
Since $\ell_n \rightarrow 0$, this proves endpoint control.
\end{proof}

\begin{remark}
An element receiving reference mass from several marked elements
inherits all their marks.
The refinement rule must ensure that these localization
constraints are compatible.
The anchors are retained isometrically,
while the remaining elements may be replaced,
so the layers need not be nested.
Moreover, the anchor chain need not be a reference path.
It provides a route for estimating the endpoint distance,
without imposing a bound on the length of the given reference path.
\end{remark}

Each of these criteria yields a compact measured limit
by Theorem~\ref{thm:compact-limit}.
For complete refinements, the limit is also geodesic
by Lemma~\ref{lem:geodesic-limit}.
These criteria provide sufficient conditions for
\eqref{eq:convergence}, under which the limit characterization
in the next section will be formulated.

\section{Limit characterization}
\label{sec:characterization}

This section compares the SCD condition with the entropic
curvature--dimension condition on the limit space, 
as formulated by
Erbar--Kuwada--Sturm \cite[Definition~3.1]{EKS2015}.
\begin{definition}[Entropic curvature--dimension condition]
\label{def:cde}
A compact metric probability space $(M,d,m)$ with $\supp m=M$
satisfies $\CDe(K,N)$ if every pair of finite-entropy probability
measures $\mu_0,\mu_1$ admits a constant-speed $W_2$-geodesic
$(\mu_t)_{0\le t\le1}$ such that, for every $t\in[0,1]$,
\begin{equation}\label{eq:cde}
 U_N(\mu_t\mid m)\ge
 \sigma_\kappa^{(1-t)}(D)U_N(\mu_0\mid m)
 +\sigma_\kappa^{(t)}(D)U_N(\mu_1\mid m),
 \text{ where } D=W_2(\mu_0,\mu_1).
\end{equation}
\end{definition}

Using the geodesicity of the limit space proved in Section~\ref{sec:limits},
we construct reference couplings 
between the refinement layers and the limit space
whose displacements tend uniformly to zero.
These couplings allow us both to pass discrete entropy comparisons
to the limit and to recover them on the refinement layers.
We also characterize the limiting condition by midpoint defects
defined on individual layers.

Throughout this section, $\Sigma$ is a complete refinement system
under the standing assumptions of Section~\ref{sec:entropy},
and \eqref{eq:convergence} holds with limit $(M,d,m)$.
We retain $R,L$, $1\le N<\infty$, and $\kappa=K/N$.
All layer indices are at least $R$, and all diameter estimates
use \eqref{eq:tail-diameter}.
By Lemmas~\ref{lem:compact} and~\ref{lem:geodesic-limit},
$Z_{\ge R}$ is compact and $M\subset Z_{\ge R}$ is geodesic.
We write $W=W_{2,Z}$ and $W_n=W_{2,(V_n,d_n)}$.
The isometric inclusions identify $W_n$ with the restriction
of $W$ to probabilities on $V_n$.
The comparison range \eqref{eq:strict-comparison-range}
remains in force when $K>0$.

\subsection{Reference couplings and entropy recovery}

We begin with a criterion for the existence 
of a coupling with a prescribed displacement bound.
\begin{lemma}[Finite-source Hall criterion]\label{lem:hall}
Let $X$ be a finite subset of a compact metric space, let $\mu$ be
a probability on $X$, and let $m$ be a probability on the ambient
space. For $r\ge0$, there exists a coupling in $\Pi(\mu,m)$
supported on $d(x,z)\le r$ if and only if
\begin{equation}\label{eq:hall}
 \mu(A)\le m\!\left(\bigcup_{x\in A}\overline B_r(x)\right)
 \quad\text{for every }A\subset X.
\end{equation}
\end{lemma}
\begin{proof}
If $\beta\in\Pi(\mu,m)$ is supported on $d(x,z)\le r$, mass
starting in $A$ can only reach $\bigcup_{x\in A}\overline B_r(x)$.
The marginal identities give \eqref{eq:hall}.

For sufficiency, reduce the support constraint to a finite network.
Partition the target into the Borel sets
\[
 F_J=\left\{z:\{x\in X:d(x,z)\le r\}=J\right\},
 \qquad J\subset X.
\]
In a finite transport network, join a source node to each $x\in X$
with capacity $\mu(x)$, join $x$ to $F_J$ with unlimited capacity
when $x\in J$, and join $F_J$ to a sink with capacity $m(F_J)$.
The target classes adjacent to a collection $A$ of source points
are those with $J\cap A\ne\varnothing$. Their union is
$\bigcup_{x\in A}\overline B_r(x)$. The associated cut has capacity
\[
 \mu(X\setminus A)+\sum_{J:J\cap A\ne\varnothing}m(F_J)
 =1-\mu(A)+m\!\left(\bigcup_{x\in A}\overline B_r(x)\right)
 \ge1.
\]
Every finite-capacity cut is bounded below by a cut of this form.
The max-flow/min-cut theorem therefore yields a flow of value one.
The total source and target capacities both equal one, so all
marginal capacities are saturated. Writing $f_{xJ}$ for
the flow from $x$ to $F_J$, we have
\[
 f_{xJ}=0\quad(x\notin J),\qquad
 \sum_Jf_{xJ}=\mu(x),\qquad \sum_xf_{xJ}=m(F_J).
\]
In particular, \eqref{eq:hall} with $A=X$ implies $m(F_\varnothing)=0$.
On each positive-mass $F_J$, distribute the flow according to the
normalized restriction of $m$ and set
\[
 \beta=\sum_{x,J:m(F_J)>0}
       f_{xJ}\delta_x\otimes\frac{m|_{F_J}}{m(F_J)}.
\]
The marginal identities give $\beta\in\Pi(\mu,m)$. Every nonzero
summand has $x\in J$ and $F_J\subset\overline B_r(x)$, which proves
the support constraint.
\end{proof}

\begin{lemma}[Quantitative upgrade from weak convergence]
\label{lem:upgrade}
Let $(Y,d_Y)$ be compact with diameter at most $L>0$, let $m$ be
a probability with connected support $M\subset Y$, and let $\mu$
be a probability measure with full support on a finite set
$X\subset Y$. Set
\[
 \alpha=d_H^Y(X,M)/L,\qquad
 \mathfrak p=d_{\mathrm{Pr},\,d_Y/L}(\mu,m),\qquad
 b(s)=\inf_{z\in M}m(B_{Ls}(z)).
\]
Here $d_{\mathrm{Pr},\,d_Y/L}$ is the Prokhorov distance for the
normalized metric. If $\eps>0$ satisfies
\begin{equation}\label{eq:upgrade-threshold}
 \alpha<\eps/4,\qquad
 \mathfrak p<\min\{\eps/4,b(\eps/8)\},
\end{equation}
then $W_{\infty,Y}(\mu,m)\le L\eps$.
Under the standing assumptions of this section, the limit is
connected. Applying this criterion in $Y=Z_{\ge R}$ therefore gives
\begin{equation}\label{eq:winfty-convergence}
 h_n:=W_{\infty,Z}(m_n,m)\longrightarrow0.
\end{equation}
\end{lemma}
\begin{proof}
Connectedness supplies the additional mass needed to absorb the
Prokhorov error. Work with $\widetilde d=d_Y/L$ and write
$A^s=\{z:\widetilde d(z,A)<s\}$. Choose
\[
 \mathfrak p<\delta<\min\{\eps/4,b(\eps/8)\}.
\]
For every nonempty $A\subset X$, Prokhorov's inequality gives
$\mu(A)\le m(A^\delta)+\delta$. If $M\subset A^\eps$, then
$m(A^\eps)=1$ and the required Hall inequality follows immediately.

Otherwise consider $f(z)=\widetilde d(z,A)$ on $M$. Choose $x\in A$.
The Hausdorff estimate supplies a point of $M$ at distance at most
$\alpha$ from $x$, so $\min_Mf\le\alpha<\eps/4$. Since
$M\not\subset A^\eps$, also $\max_Mf\ge\eps$. Continuity and
connectedness imply that $f(z_0)=\eps/2$ for some $z_0\in M$.
As $f$ is $1$-Lipschitz, every $z\in B_{\eps/8}^{\widetilde d}(z_0)$
satisfies
\[
 \frac{3\eps}{8}<f(z)<\frac{5\eps}{8}.
\]
The ball lies in $A^\eps$ and outside $A^\delta$, since
$\delta<\eps/4$. Its mass is at least $b(\eps/8)>\delta$, enough
to absorb the Prokhorov error. The two disjoint subsets of $A^\eps$
therefore give
\[
 \mu(A)\le m(A^\delta)+\delta
 \le m(A^\delta)+m(B_{\eps/8}^{\widetilde d}(z_0))
 \le m(A^\eps).
\]
The case $A=\varnothing$ is immediate. Since the open neighborhood
$A^\eps$ is contained in the union of the closed $\eps$-balls,
Lemma~\ref{lem:hall} gives a coupling with displacement
at most $L\eps$ in $d_Y$.

It remains to show that $b(s)>0$ for every $s>0$.
Choose finitely many points
$z_1,\ldots,z_q\in M$ whose open $Ls/2$-balls cover $M$. Full support
gives $m(B_{Ls/2}(z_i))>0$. For every $z\in M$, select $i$ with
$d_Y(z,z_i)<Ls/2$; then $B_{Ls/2}(z_i)\subset B_{Ls}(z)$, whence
\[
 b(s)\ge\min_{1\le i\le q}m(B_{Ls/2}(z_i))>0.
\]
Apply the criterion in $Y=Z_{\ge R}$. Under \eqref{eq:convergence},
the corresponding Hausdorff and Prokhorov errors tend to zero.
For each fixed $\eps>0$, the threshold
\eqref{eq:upgrade-threshold} eventually holds, proving
\eqref{eq:winfty-convergence}.
\end{proof}

The convergence \eqref{eq:winfty-convergence} uses only the
refinement geometry and \eqref{eq:convergence}.
The infimum defining $h_n$ is attained: choose couplings
with support radii at most $h_n+1/j$ and extract a weakly converging
subsequence in the compact coupling space. For every $\eta>0$, all
sufficiently late couplings are concentrated on the closed set
$\{(x,z):d_Z(x,z)\le h_n+\eta\}$. The limit is concentrated there
as well. Intersecting over $\eta\downarrow0$ gives support radius
$h_n$.
We may therefore choose $\beta_n\in\Pi(m_n,m)$ supported on
$d_Z(x,z)\le h_n$.
These couplings are constructed directly from the reference measures.
For an endpoint-controlled system with endpoint modulus $\omega_n$,
Theorem~\ref{thm:compact-limit} gives $h_n\le\omega_n$.
The couplings constructed there have support radii at most $\omega_n$.
The transfer and recovery arguments below also apply to these couplings, 
with $\omega_n$ in place of $h_n$.

A reference coupling transfers probabilities by averaging densities.
For any $\beta_n\in\Pi(m_n,m)$, write
\[
 \beta_n(\{x\},\dd z)=m_n(x)k_n(x,\dd z)=p_n(z,x)m(\dd z).
\]
Here $k_n$ and $p_n$ are stochastic kernels. Define the lift $E_n$
and the projection $T_n$ by
\begin{align}
 E_n(\rho_nm_n)(\dd z)
 &=\sum_x\rho_n(x)\beta_n(\{x\},\dd z),\label{eq:lift}\\
 T_n(\rho m)(x)&=\int p_n(z,x)\rho(z)\,\dd m(z).
 \label{eq:project}
\end{align}
Both maps preserve probability and depend on the chosen coupling.

\begin{lemma}[Transfer and entropy contraction]\label{lem:transfer}
For every $p\in\Pp(V_n)$ and finite-entropy $\mu\ll m$,
\begin{equation}\label{eq:entropy-contraction}
 \Ent_m(E_np)\le H_n(p),\qquad H_n(T_n\mu)\le\Ent_m(\mu).
\end{equation}
If $\beta_n$ is supported on $d_Z(x,z)\le h$, then, uniformly over
these probabilities,
\begin{equation}\label{eq:uniform-transfer}
 W(E_np,p)\le h,\qquad W(T_n\mu,\mu)\le h.
\end{equation}
\end{lemma}
\begin{proof}
Write $p=\rho_nm_n$ and $\mu=\rho m$. The transferred densities are
\[
 \frac{\dd E_np}{\dd m}(z)=\sum_xp_n(z,x)\rho_n(x),\qquad
 \frac{T_n\mu(x)}{m_n(x)}=\int\rho(z)k_n(x,\dd z).
\]
Each density is a conditional average of the source density.
The kernel entropy inequality \eqref{eq:kernel-entropy-contraction},
applied to $k_n$ and $p_n$, gives
\eqref{eq:entropy-contraction}.

For the transport estimates, define
\[
 \beta_n^p(\{x\},\dd z)=\rho_n(x)\beta_n(\{x\},\dd z),\qquad
 \beta_n^\mu(\{x\},\dd z)=\rho(z)\beta_n(\{x\},\dd z).
\]
These measures have marginals $(p,E_np)$ and $(T_n\mu,\mu)$,
respectively, and are supported on $\supp\beta_n$.
If $d_Z(x,z)\le h$ there, both transport costs are at most $h^2$.
This proves \eqref{eq:uniform-transfer}.
\end{proof}

To recover entropy values, we also use joint lower semicontinuity.
On compact $Z$, the variational formula~\cite{Villani2009}
\[
 \Ent_\lambda(\mu)=\sup_{f\in C(Z)}
 \left\{\int f\,\dd\mu-\log\int e^f\,\dd\lambda\right\}
\]
expresses relative entropy as a supremum of continuous functions
of $(\mu,\lambda)$.
Thus relative entropy is jointly lower semicontinuous,
and $U_N(\mu\mid\lambda)$ is jointly upper semicontinuous.

\begin{lemma}[Fixed-endpoint entropy recovery]
\label{lem:recovery}
Choose reference couplings $\beta_n\in\Pi(m_n,m)$ supported on
$d_Z(x,z)\le h_n$, where $h_n\to0$, and let $T_n$ be the
corresponding projections.
For every fixed finite-entropy probability $\mu\ll m$,
\begin{equation}
 W(T_n\mu,\mu)\longrightarrow0,
 \qquad H_n(T_n\mu)\longrightarrow\Ent_m(\mu).
 \label{eq:recovery}
\end{equation}
\end{lemma}
\begin{proof}
Lemma~\ref{lem:transfer} gives $W(T_n\mu,\mu)\le h_n\to0$.
Since $m_n\rightharpoonup m$, joint lower semicontinuity and
entropy contraction give
\[
 \Ent_m(\mu)
 \le\liminf_n H_n(T_n\mu)
 \le\limsup_n H_n(T_n\mu)
 \le\Ent_m(\mu).
\]
This proves the entropy convergence.
\end{proof}

The transport estimate is uniform over the source probabilities.
The entropy convergence is asserted for each fixed finite-entropy
endpoint. For endpoints that vary with the layer, we will use
entropy contraction directly.

\subsection{Midpoint defects on individual layers}

Recall the midpoint coefficient $S_\kappa$ from
\eqref{eq:mid-coefficient}. For $n\ge R$ and
$p,q,\nu\in\Pp(V_n)$, set
\[
 D=W_n(p,q),\quad
 A_n(p,q;\nu)=2W_n(p,\nu)^2+2W_n(\nu,q)^2,
 \quad F_n(p,q)=S_\kappa(D)[U_n(p)+U_n(q)].
\]
The single-layer midpoint defect is
\begin{equation}\label{eq:single-defect}
 \Delta_n(K,N)=\max_{p,q\in\Pp(V_n)}\ \min_{\nu\in\Pp(V_n)}
 \max\left\{\frac{A_n(p,q;\nu)-D^2}{L^2},
                   \bigl(F_n(p,q)-U_n(\nu)\bigr)_+\right\}.
\end{equation}
The defect is nonnegative, since $A_n\ge D^2$.
Both extrema are attained: the probability simplex is compact,
$u\log u$ extends continuously to zero, and the Wasserstein distance
is continuous. Uniform continuity on the compact product also makes
the minimum value continuous in $(p,q)$.

Thus $\Delta_n(K,N)\to0$ means that both midpoint errors vanish
uniformly over the endpoint probabilities.
This is the single-layer counterpart of the admissibility and
entropy requirements in Definition~\ref{def:midpoint}:
the endpoints and the intermediate measure now all lie on $V_n$.

\begin{proposition}[Obstruction on a finite layer]
\label{prop:finite-obstruction}
If $|V_n|\ge2$, there exists a pair of finite-entropy probability measures on $X_n$ 
which admit no midpoint with respect to $W_n$. 
Moreover,
\[
 \Delta_n(K,N)\ge \delta_n^2/L^2,
 \qquad
 \delta_n=\min_{x\ne y}d_n(x,y).
\]
Consequently, $X_n$ satisfies 
neither $\CDe(K,N)$ 
nor the classical Lott--Sturm--Villani condition $\mathrm{CD}(K,N)$, 
both understood with respect to the metric Wasserstein distance $W_n$.
\end{proposition}
\begin{proof}
Choose $x,y\in V_n$ with $d_n(x,y)=\delta_n$.
For every $z\in V_n$,
$d_n(x,z)^2+d_n(z,y)^2\ge\delta_n^2$: equality holds at $x,y$,
whereas both terms are at least $\delta_n^2$ at every other point.
Consequently every $\nu\in\Pp(V_n)$ satisfies
\[
 W_n(\delta_x,\nu)^2+W_n(\nu,\delta_y)^2\ge\delta_n^2.
\]
A Wasserstein midpoint would instead make this sum $\delta_n^2/2$.
The endpoint Dirac measures have finite entropy because the reference
masses are positive. Thus the required comparison geodesic does not
exist. Multiplying the inequality by two and subtracting
$D^2=\delta_n^2$ proves the proposition.
\end{proof}
This obstruction concerns the metric Wasserstein distance $W_n$.
Curvature conditions for graphs and Markov chains 
may instead use a different transport metric or a chosen generator \cite{ErbarMaas2012}.
Here we retain the metric Wasserstein distances and impose the
curvature condition on the refinement sequence.
Although each nontrivial finite layer has a positive midpoint defect,
these defects may tend to zero along the sequence.
Under the standing convergence assumptions, 
we shall show that this vanishing is equivalent to $\SCD(K,N)$ for the sequence
and to $\CDe(K,N)$ for its limit.

\subsection{Quantitative recovery on layers and segments}

In this subsection, suppose that the limit $M$ satisfies $\CDe(K,N)$.
We lift the prescribed endpoint probability measures 
from their respective refinement layers to $M$.
We then choose a comparison geodesic between the lifted measures 
and project its intermediate measures onto the required layers.
Entropy contraction preserves the lower comparison bounds;
the support radii control the transport and coefficient errors.

\begin{theorem}[Single-layer midpoint recovery]\label{thm:single-recovery}
Suppose $(M,d,m)$ satisfies $\CDe(K,N)$ and
$\beta_n\in\Pi(m_n,m)$ is supported on $d_Z(x,z)\le h$.
For any $p,q\in\Pp(V_n)$ with $D=W_n(p,q)$, there exists
$\nu\in\Pp(V_n)$ such that, writing $a=W_n(p,\nu)$ and $b=W_n(\nu,q)$,
\begin{align}
 |a-D/2|, |b-D/2|&\le3h,\label{eq:midpoint-position}\\
 0\le 2a^2+2b^2-D^2
 &\le\min\{12Lh+16h^2,\ 12Dh+36h^2\}.\label{eq:sharp-action}
\end{align}
Define
\begin{equation}\label{eq:Ck}
 C_{\kappa,L}=
 \begin{cases}
 \displaystyle\frac{\kappa}{8\cos^2(\sqrt\kappa L/2)},&\kappa>0,\\
 0,&\kappa=0,\\
 |\kappa|/8,&\kappa<0.
 \end{cases}
\end{equation}
The same $\nu$ satisfies
\begin{equation}\label{eq:sharp-entropy}
 U_n(\nu)\ge
 \left[S_\kappa(D)-2C_{\kappa,L}h(D+h)\right]
          [U_n(p)+U_n(q)].
\end{equation}
For $D=0$ one may take $\nu=p=q$, giving exact comparisons.
\end{theorem}
\begin{proof}
Assume $D>0$. Set
$\bar p=E_np$ and $\bar q=E_nq$, and write
$D_*=W_2(\bar p,\bar q)$. Lemma~\ref{lem:transfer} gives
finite entropy for both lifts and
\[
 W(p,\bar p)\le h,\qquad W(q,\bar q)\le h,
 \qquad |D_*-D|\le2h.
\]
Choose a midpoint $\bar\nu$ on a $\CDe(K,N)$ comparison geodesic
between the lifts. The comparison gives $U_N(\bar\nu\mid m)>0$,
so $\bar\nu$ has finite entropy. Put $\nu=T_n\bar\nu$; then
$W(\nu,\bar\nu)\le h$. The reverse triangle inequality yields
\begin{align*}
 \left|W(p,\nu)-W(\bar p,\bar\nu)\right|
 &\le W(p,\bar p)+W(\nu,\bar\nu)\le2h,\\
 \left|W(\nu,q)-W(\bar\nu,\bar q)\right|
 &\le W(\nu,\bar\nu)+W(q,\bar q)\le2h.
\end{align*}
Since both lifted distances equal $D_*/2$, we obtain
$|a-D_*/2|,|b-D_*/2|\le2h$. Combining with $|D_*-D|\le2h$
proves \eqref{eq:midpoint-position} and $a,b\le D/2+3h$.

For the first action estimate, squaring $a,b\le D_*/2+2h$ gives
\[
 2a^2+2b^2\le4(D_*/2+2h)^2
             =D_*^2+8D_*h+16h^2.
\]
Since $D,D_*\le L$, we have
$|D_*^2-D^2|\le2L|D_*-D|\le4Lh$. It follows that
$2a^2+2b^2-D^2\le12Lh+16h^2$.
Squaring $a,b\le D/2+3h$ instead gives
\[
 2a^2+2b^2\le4(D/2+3h)^2
             =D^2+12Dh+36h^2.
\]
Finally $D\le a+b$ implies
$D^2\le(a+b)^2\le2(a^2+b^2)$, proving nonnegativity.

For the entropy estimate, we bound the variation of $S_\kappa$. If $\kappa>0$, set
$u=\sqrt\kappa s/2$. Differentiation gives
\[
 S_\kappa'(s)=\frac{\sqrt\kappa\sin u}{4\cos^2u},\qquad
 |S_\kappa'(s)|\le
 \frac{\kappa s}{8\cos^2(\sqrt\kappa L/2)}.
\]
Here $0\le u\le\sqrt\kappa L/2<\pi/2$, so the denominator is
positive and $\sin u\le u$. If $\kappa<0$, the corresponding formula is
\[
 |S_\kappa'(s)|=
 \frac{\sqrt{-\kappa}\sinh u}{4\cosh^2u}
 =\frac{\sqrt{-\kappa}}4\frac{\tanh u}{\cosh u}
 \le\frac{|\kappa|s}{8},
 \qquad u=\sqrt{-\kappa}s/2.
\]
For $\kappa=0$ the derivative is zero. Thus in every case
$|S_\kappa'(s)|\le C_{\kappa,L}s$ on $[0,L]$. Integrating between
$D$ and $D_*$ and using $D_*\le D+2h$ gives
\begin{align*}
 |S_\kappa(D_*)-S_\kappa(D)|
 &\le\frac{C_{\kappa,L}}2|D_*^2-D^2|\\
 &\le\frac{C_{\kappa,L}}2(2h)(2D+2h)
 =2C_{\kappa,L}h(D+h).
\end{align*}
Apply entropy contraction before and after the limit midpoint
comparison to obtain
\begin{align*}
 U_n(\nu)
 &\ge U_N(\bar\nu\mid m)\\
 &\ge S_\kappa(D_*)[U_N(\bar p\mid m)+U_N(\bar q\mid m)]\\
 &\ge S_\kappa(D_*)[U_n(p)+U_n(q)].
\end{align*}
The last inequality uses $S_\kappa(D_*)>0$.
The coefficient estimate proves \eqref{eq:sharp-entropy}; it remains
valid when the bracket on its right-hand side is negative. 
When $D=0$, taking $\nu=p=q$ gives zero action and equality
in the unperturbed midpoint entropy comparison,
since $S_\kappa(0)=1/2$.
\end{proof}

For the couplings chosen with support radius $h_n$,
apply Theorem~\ref{thm:single-recovery} with $h=h_n$.
Since $h_n\to0$, we obtain
\begin{equation}\label{eq:defect-bound}
 \Delta_n(K,N)\le
 \max\left\{\frac{12Lh_n+16h_n^2}{L^2},
                   4C_{\kappa,L}h_n(L+h_n)\right\}\longrightarrow0.
\end{equation}
Here the entropy bound uses $U_n(p)+U_n(q)\le2$.
When $K=0$, the entropy comparison holds without an error term;
the action error may still be positive.

\begin{proposition}[Uniform recovery on fixed segments]
\label{prop:window-recovery}
Suppose $(M,d,m)$ satisfies $\CDe(K,N)$ and, for every $n\ge R$,
there is a coupling $\beta_n\in\Pi(m_n,m)$ supported on
$d_Z(x,z)\le h_n$, where $h_n\to0$. Put $\widehat h_r=\sup_{n\ge r}h_n$ and
\[
 B_{\kappa,L}=\sup_{0\le t\le1,\,0\le s\le L}
                   |\partial_s\sigma_\kappa^{(t)}(s)|<\infty.
\]
For each fixed integer $J\ge2$, every $r\ge R$, and every pair
$\mu_0\in\Pp(V_r)$, $\mu_J\in\Pp(V_{r+J})$, there is a probability
section whose action excess is at most
\begin{equation}\label{eq:window-action-rate}
 4(J+1)L\widehat h_r+4J^2\widehat h_r^2
\end{equation}
and whose entropy comparison error is at most
$4B_{\kappa,L}\widehat h_r$ at every intermediate layer.
For every $\eps>0$ and all sufficiently large $r$, this section
belongs to $\mathcal G_{r,J}^{\eps}(\mu_0,\mu_J)$ and satisfies
all entropy comparisons with tolerance $\eps$.
In particular, $\Sigma$ satisfies $\SCD(K,N)$.
\end{proposition}
\begin{proof}
Fix $r,J$ and endpoints
$\mu_0\in\Pp(V_r)$,
$\mu_J\in\Pp(V_{r+J})$. Set $h=\widehat h_r$ and lift the endpoints
to $\bar\mu_0=E_r\mu_0$ and $\bar\mu_1=E_{r+J}\mu_J$.
Choose one $\CDe(K,N)$ comparison geodesic $(\bar\mu_t)_{t\in[0,1]}$
between them, and put
\[
 \mu_j=T_{r+j}\bar\mu_{j/J}\quad(1\le j<J),\qquad
 D_*=W_2(\bar\mu_0,\bar\mu_1).
\]
Keep the original endpoint measures and write
$\boldsymbol\mu=(\mu_0,\ldots,\mu_J)$ and $D=W(\mu_0,\mu_J)$.
The comparison inequality ensures that 
every $\bar\mu_t$ has finite entropy.
Lemma~\ref{lem:transfer} gives
$W(\mu_j,\bar\mu_{j/J})\le h$ for $0\le j\le J$, and hence
$|D-D_*|\le2h$. Therefore
\begin{align*}
 W(\mu_j,\mu_{j+1})
 &\le W(\mu_j,\bar\mu_{j/J})
     +W(\bar\mu_{j/J},\bar\mu_{(j+1)/J})\\
 &\qquad+W(\bar\mu_{(j+1)/J},\mu_{j+1})
 \le D_*/J+2h.
\end{align*}
By the definition of the action,
\begin{align*}
 \mathcal A_{r,J}(\boldsymbol\mu)
 &\le J^2(D_*/J+2h)^2
   =D_*^2+4JD_*h+4J^2h^2\\
 &\le D^2+4(J+1)Lh+4J^2h^2,
\end{align*}
where $D_*\le L$ and $|D_*^2-D^2|\le4Lh$.

For $t=j/J$, entropy contraction followed by the comparison along
this same geodesic gives
\begin{align*}
 U_{r+j}(\mu_j)
 &\ge U_N(\bar\mu_t\mid m)\\
 &\ge\sigma_\kappa^{(1-t)}(D_*)U_N(\bar\mu_0\mid m)
      +\sigma_\kappa^{(t)}(D_*)U_N(\bar\mu_1\mid m)\\
 &\ge\sigma_\kappa^{(1-t)}(D_*)U_r(\mu_0)
      +\sigma_\kappa^{(t)}(D_*)U_{r+J}(\mu_J).
\end{align*}
The last inequality uses nonnegativity of the coefficients.
Their uniform derivative bound gives
$|\sigma_\kappa^{(t)}(D_*)-\sigma_\kappa^{(t)}(D)|
\le B_{\kappa,L}|D_*-D|\le2B_{\kappa,L}h$.
Replacing the two coefficients therefore gives a total error at most
$4B_{\kappa,L}h$, since $U_r(\mu_0),U_{r+J}(\mu_J)\le1$.
Fix $J\ge2$ and $\eps>0$. Since $\widehat h_r\to0$, choose
$r_0\ge R$ so that, for every $r\ge r_0$,
\[
 4(J+1)L\widehat h_r+4J^2\widehat h_r^2\le\eps L^2,
 \qquad 4B_{\kappa,L}\widehat h_r\le\eps.
\]
The same section then belongs to
$\mathcal G_{r,J}^{\eps}(\mu_0,\mu_J)$ and satisfies the entropy
comparison at every intermediate layer.
The threshold is independent of the endpoint probabilities,
which proves $\SCD(K,N)$.
\end{proof}

Theorem~\ref{thm:admissible-existence} constructs geometrically
admissible sections from the refinement data alone.
The present recovery also controls entropy, using the comparison
geodesic in the limit.
If a criterion in Section~\ref{sec:limits} gives an endpoint
modulus $\omega_n\le C\ell_n$ on a tail, 
then $h_n\le\omega_n\le C\ell_n$.
This bound can be inserted into
\eqref{eq:defect-bound} and \eqref{eq:window-action-rate}.

\subsection{The limit characterization}

To derive $\CDe(K,N)$ from either $\SCDmid(K,N)$ or the condition $\Delta_n(K,N)\to0$, 
we pass the discrete midpoint comparisons to the limit.
This yields an exact midpoint comparison on $M$, 
which the following lemma extends to all times along a single Wasserstein geodesic.

\begin{lemma}[From midpoint comparisons to the entropic condition]
\label{lem:dyadic}
Suppose every pair of finite-entropy probabilities on $(M,d,m)$
admits a finite-entropy $W_2$-midpoint satisfying the comparison
\eqref{eq:cde} at $t=1/2$.
Then $(M,d,m)$ satisfies $\CDe(K,N)$.
\end{lemma}
\begin{proof}
We first construct a geodesic at dyadic times, then extend both the
metric identity and the entropy comparison to all times.
Fix finite-entropy endpoints $\mu_0,\mu_1$ and write
$D=W_2(\mu_0,\mu_1)$. If $D=0$, the constant curve suffices.
Suppose $D>0$. Select a comparison midpoint at time $1/2$, then
midpoints between the measures at times $0,1/2$ and $1/2,1$, and
continue recursively. The comparison gives a positive value of $U_N$ at each new midpoint.
It therefore has finite entropy and can be used as an endpoint in
the next step.
The construction assigns consistent measures to all dyadic times,
and consecutive measures at depth $k$ have distance $2^{-k}D$.

For dyadic $a<b$, subdivide to a common depth. The chain between $a$
and $b$ has length $(b-a)D$, while the chains from $0$ to $a$ and
from $b$ to $1$ have lengths $aD$ and $(1-b)D$. Consequently
\[
 W_2(\mu_a,\mu_b)\le(b-a)D,\qquad
 D\le aD+W_2(\mu_a,\mu_b)+(1-b)D.
\]
Thus $W_2(\mu_a,\mu_b)=(b-a)D$. For any $t\in[0,1]$, a sequence
of dyadic times $t_k\to t$ gives a Cauchy sequence $\mu_{t_k}$.
Completeness of $(\Pp(M),W_2)$ gives a limit $\mu_t$.
The distance identity makes this limit independent of the chosen
dyadic approximation. Passing to the limit in that identity yields
$W_2(\mu_s,\mu_t)=|s-t|D$ for all $s,t$. Thus the extension is a
constant-speed geodesic.

To propagate the entropy bound, define
\[
 \phi(t)=\sigma_\kappa^{(1-t)}(D)U_N(\mu_0\mid m)
          +\sigma_\kappa^{(t)}(D)U_N(\mu_1\mid m).
\]
It solves $\phi''+\kappa D^2\phi=0$. On any subinterval $[a,b]$,
the comparison function with endpoint values $\phi(a),\phi(b)$
solves the same equation. The boundary-value problem has a unique
solution: this follows from $(b-a)D\sqrt\kappa<\pi$ when
$\kappa>0$, and from the linear or hyperbolic form of the solutions
when $\kappa\le0$. Hence
\[
 \phi((a+b)/2)=S_\kappa((b-a)D)[\phi(a)+\phi(b)].
\]
At the first dyadic level the desired bound is the assumed midpoint
comparison. If it holds at one level, positivity of $S_\kappa$ and
the displayed identity imply it at every newly inserted midpoint.
Thus $U_N(\mu_t\mid m)\ge\phi(t)$ for all dyadic $t$. For arbitrary
$t$, upper semicontinuity and continuity of $\phi$ give
\[
 U_N(\mu_t\mid m)\ge\limsup_kU_N(\mu_{t_k}\mid m)
                         \ge\lim_k\phi(t_k)=\phi(t).
\]
The lower bound is positive, so all measures on the resulting
geodesic have finite entropy. This proves \eqref{eq:cde}.

\end{proof}

\begin{theorem}[Limit characterization]
\label{thm:equivalence}\label{thm:characterization}
Under the assumptions of this section, the following conditions
are equivalent:
\begin{enumerate}
\item $\Sigma$ satisfies $\SCD(K,N)$;
\item $\Sigma$ satisfies $\SCDmid(K,N)$;
\item $\Delta_n(K,N)\to0$;
\item $(M,d,m)$ satisfies $\CDe(K,N)$.
\end{enumerate}
\end{theorem}
\begin{proof}
Theorem~\ref{thm:direct-midpoint} gives (1)$\Leftrightarrow$(2)
without any convergence assumption. We prove
(2)$\Rightarrow$(4) and (3)$\Rightarrow$(4) by recovering fixed
endpoints and passing midpoint comparisons to the limit.
Then uniform recovery gives (4)$\Rightarrow$(3) and
(4)$\Rightarrow$(1).
Fix reference couplings with support radii $h_n\to0$ as in
\eqref{eq:winfty-convergence}, and let $E_n,T_n$ be their transfer maps.

Assume (2) and fix finite-entropy endpoints $\mu_0,\mu_1$ on $M$.
Let $\eps_i\downarrow0$ and choose increasing
$r_i\to\infty$ beyond the midpoint threshold for $\eps_i$. Put
\[
 p_i=T_{r_i}\mu_0,\qquad q_i=T_{r_i+2}\mu_1,
 \qquad D_i=W(p_i,q_i).
\]
Lemma~\ref{lem:recovery} gives
\begin{gather*}
 p_i\longrightarrow\mu_0,\quad q_i\longrightarrow\mu_1
 \quad\text{in }W_2,\qquad D_i\longrightarrow D:=W_2(\mu_0,\mu_1),\\
 U_{r_i}(p_i)\longrightarrow U_N(\mu_0\mid m),\qquad
 U_{r_i+2}(q_i)\longrightarrow U_N(\mu_1\mid m).
\end{gather*}
The midpoint condition gives an $\eps_i$-admissible midpoint
$\nu_i\in\Pp(V_{r_i+1})$ with
\begin{align*}
 2W(p_i,\nu_i)^2+2W(\nu_i,q_i)^2
 &\le D_i^2+\eps_iL^2,\\
 U_{r_i+1}(\nu_i)
 &\ge S_\kappa(D_i)[U_{r_i}(p_i)+U_{r_i+2}(q_i)]-\eps_i.
\end{align*}
Compactness of $Z_{\ge R}$ gives a subsequence converging in $W_2$
to a probability $\nu$. It is supported on $M$: indeed
\[
 \int d_Z(z,M)^2\,\dd\nu_i(z)\le a_{r_i+1}^2\longrightarrow0,
\]
and the integrand is continuous on $Z$. Passing to the limit in the
action inequality gives
\[
 2W_2(\mu_0,\nu)^2+2W_2(\nu,\mu_1)^2\le D^2.
\]
Weak convergence of $m_{r_i+1}$ to $m$ and joint upper
semicontinuity of $U_N$ give
\begin{align*}
 U_N(\nu\mid m)
 &\ge\limsup_i U_{r_i+1}(\nu_i)\\
 &\ge S_\kappa(D)[U_N(\mu_0\mid m)+U_N(\mu_1\mid m)].
\end{align*}
Write $a=W_2(\mu_0,\nu)$ and $b=W_2(\nu,\mu_1)$. The action bound
and the triangle inequality imply
\[
 D^2\le(a+b)^2\le2(a^2+b^2)\le D^2.
\]
All inequalities are equalities. Since
$2(a^2+b^2)-(a+b)^2=(a-b)^2$, this forces $a=b=D/2$.
The entropy lower bound is positive, including when $D=0$, so $\nu$
is a finite-entropy comparison midpoint. Lemma~\ref{lem:dyadic}
proves (4).

Assume (3), and again fix finite-entropy endpoints $\mu_0,\mu_1$
on $M$. Recover them on the same layer as
$p_n=T_n\mu_0$ and $q_n=T_n\mu_1$. A measure $\nu_n$ attaining the inner minimum in
\eqref{eq:single-defect} satisfies
\begin{align*}
 2W_n(p_n,\nu_n)^2+2W_n(\nu_n,q_n)^2
 &\le W_n(p_n,q_n)^2+L^2\Delta_n,\\
 U_n(\nu_n)&\ge
 S_\kappa(W_n(p_n,q_n))[U_n(p_n)+U_n(q_n)]-\Delta_n.
\end{align*}
Both errors tend to zero. The preceding compactness and
semicontinuity argument gives a limit midpoint satisfying the
entropy comparison. Lemma~\ref{lem:dyadic} then gives (4).

Finally, assume (4). For the fixed reference couplings,
Theorem~\ref{thm:single-recovery} and \eqref{eq:defect-bound}
give $\Delta_n(K,N)\to0$, proving (3).
Proposition~\ref{prop:window-recovery} gives an admissible section
with all entropy comparisons on each sufficiently late segment
of every fixed length, proving (1).
\end{proof}

Convergence thus relates the geometry of the refinement sequence to that of its limit.
Under the standing assumptions, 
the sequence satisfies $\SCD(K,N)$ if and only if its limit satisfies $\CDe(K,N)$.
If, in addition, $N>1$ and the limit is essentially nonbranching, 
these conditions are also equivalent to the classical $\mathrm{CD}(K,N)$ condition, 
as shown below.
\begin{corollary}[Classical CD under essential nonbranching]
\label{cor:classical-cd}
Under the hypotheses of Theorem~\ref{thm:characterization}, assume
in addition that $N>1$ and that $(M,d,m)$ is essentially nonbranching. Then the four conditions
in that theorem are also equivalent to the classical condition
$\mathrm{CD}(K,N)$, with the same parameters.
\end{corollary}
\begin{proof}
Essential nonbranching means that every optimal dynamical transport
plan between absolutely continuous probabilities is concentrated on
a nonbranching set of geodesics. In this setting,
Erbar--Kuwada--Sturm \cite[Theorem~3.12]{EKS2015} prove
$\CDe(K,N)\Longleftrightarrow\mathrm{CD}^{*}(K,N)$.
Since $m$ is a probability measure, the globalization theorem of
Cavalletti--Milman \cite[Main Theorem~1.1]{CavallettiMilman2021}
gives $\mathrm{CD}^{*}(K,N)\Longleftrightarrow\mathrm{CD}(K,N)$.
The conclusion follows from Theorem~\ref{thm:characterization}.
\end{proof}

\section{Geometric and entropy estimates}
\label{sec:quantitative-recovery}

Section~\ref{sec:characterization} relates the SCD condition to
vanishing midpoint defects on individual layers.
We now estimate these defects in terms of geometric approximation
and reference mass accuracy, and determine the transport scales
on which the relative errors vanish.

We retain the standing assumptions and notation of
Section~\ref{sec:characterization}, and assume that the limit
$(M,d,m)$ satisfies $\CDe(K,N)$.
Thus $\Sigma$ is a complete refinement system satisfying
\eqref{eq:convergence}, with $1\le N<\infty$ and $\kappa=K/N$.
All layer indices satisfy $n\ge R$, and the diameter and comparison
bounds are those fixed in Section~\ref{sec:entropy}.
We write $U_N(\mu)=U_N(\mu\mid m)$ and retain
\[
 a_n=d_H^Z(V_n,M),\qquad h_n=W_{\infty,Z}(m_n,m).
\]
Distances between layers and the limit are computed with $d_Z$.
Compactness ensures that the infimum defining $h_n$ is attained.
All balls below are open.

\subsection{Ball masses and reference-measure approximation}

The limiting curvature condition gives a uniform lower bound on ball
masses. This makes the coupling estimate of Lemma~\ref{lem:upgrade}
quantitative. Set
\begin{equation}
 c_{K,N,L}=\begin{cases}
 1,&K\geq0,\\[1mm]
 \displaystyle\frac{L\sqrt{-K/N}}{\sinh(L\sqrt{-K/N})},&K<0.
 \end{cases}
 \label{eq:mass-growth-constant}
\end{equation}
Then $0<c_{K,N,L}\leq1$.

\begin{theorem}[A lower bound on ball masses]
\label{thm:ball-mass}
For every $x\in M$ and $0<r\leq L$,
\begin{equation}
 m(B_r(x))\geq
 \left(c_{K,N,L}\frac rL\right)^N.
 \label{eq:automatic-ball-mass}
\end{equation}
The lower mass bound therefore follows from the limiting curvature
condition and requires no additional assumption on the refinement.
\end{theorem}

\begin{proof}
We first record a localization property of Wasserstein geodesics.
Let $(\mu_s)_{0\leq s\leq1}$ be a constant-speed $W_2$ geodesic on $M$.
Write $D=W_2(\mu_0,\mu_1)$ and fix $t\in(0,1)$.
Glue optimal couplings between times $0,t$ and $t,1$ to obtain a
probability measure on triples $(u,z,v)$. On this space put
$a=d(u,z)$, $b=d(z,v)$ and $c=d(u,v)$. The endpoint marginal is a
coupling of $\mu_0$ and $\mu_1$, while the two chosen couplings are
optimal. Consequently,
\[
 D\leq\|c\|_2\leq\|a+b\|_2
 \leq\|a\|_2+\|b\|_2=tD+(1-t)D=D.
\]
Equality holds throughout. Since $0\leq c\leq a+b$, equality
of the $L^2$ norms gives $c=a+b$ almost surely. If $D>0$, equality
in the $L^2$ triangle inequality makes $a$ and $b$ proportional;
their norms $tD$ and $(1-t)D$ therefore give
\[
 a=tc,\qquad b=(1-t)c\qquad\text{almost surely}.
\]
If $D=0$, all three quantities vanish. In either case $d(u,z)\leq tL$.
Thus, if $\mu_0$ is concentrated on $B_\eps(x)$, then
$d(x,z)\leq d(x,u)+d(u,z)<\eps+tL$ almost surely, and $\mu_t$
is concentrated on $B_{\eps+tL}(x)$.

Now fix $0<t<r/L$ and choose $\eps>0$ with $\eps+tL<r$.
Full support gives $m(B_\eps(x))>0$, so
\[
 \mu_0=\frac{m|_{B_\eps(x)}}{m(B_\eps(x))},\qquad \mu_1=m
\]
are finite-entropy probabilities. Indeed,
$\Ent_m(\mu_0)=-\log m(B_\eps(x))$ and $\Ent_m(m)=0$.
Choose a comparison geodesic supplied by $\CDe(K,N)$. Since
$U_N(m)=1$ and the other endpoint term is nonnegative, its entropy
inequality gives
\[
 U_N(\mu_t)\geq\sigma_\kappa^{(t)}(D),
 \qquad D=W_2(\mu_0,m)\leq L.
\]
For $\kappa>0$, put $a=\sqrt\kappa D<\pi$.
Concavity of sine on $[0,a]$ gives $\sin(ta)\geq t\sin a$,
hence $\sigma_\kappa^{(t)}(D)\geq t$, with its continuous value at
$D=0$. For $\kappa=0$ equality holds.
For $\kappa<0$, put $a=\sqrt{-\kappa}D$ and
$A=\sqrt{-\kappa}L$. The inequalities $\sinh(ta)\geq ta$ and
$a/\sinh a\geq A/\sinh A$ give
\[
 \sigma_\kappa^{(t)}(D)=\frac{\sinh(ta)}{\sinh a}
 \geq t\frac a{\sinh a}
 \geq t\frac A{\sinh A}=c_{K,N,L}t.
\]
Here $a/\sinh a$ is decreasing on $[0,\infty)$, since
$\sinh a-a\cosh a\leq0$; values at zero are understood by continuity.

The localization property shows that $\mu_t$ is concentrated on $B_r(x)$.
Its entropy power is bounded above by the $N$th root of the mass
of this ball.
For any probability $\nu$ concentrated on a measurable set $B$ with
$m(B)>0$, comparison with $m_B=m|_B/m(B)$ gives
\[
 \Ent_m(\nu)=\Ent_{m_B}(\nu)-\log m(B)\geq-\log m(B),
 \qquad U_N(\nu)\leq m(B)^{1/N}.
\]
The inequality follows from Jensen's inequality for $s\log s$;
if $\nu$ is not absolutely continuous, its entropy is infinite and
its entropy power is zero. We conclude that
$m(B_r(x))^{1/N}\geq c_{K,N,L}t$. Letting $t\uparrow r/L$
proves the bound for the open ball itself.
\end{proof}

\begin{theorem}[A rate from Hausdorff and Prokhorov approximation]
\label{thm:prokhorov-rate}
Let
\[
 \mathfrak p_n=d_{\mathrm{Pr},\,d_Z/L}(m_n,m)
\]
be the Prokhorov distance in the normalized intrinsic bridge metric.
Then $0\leq\mathfrak p_n\leq1$, and
\begin{equation}
 h_n\leq\min\left\{L,\;
 4a_n+\frac{8L}{c_{K,N,L}}\mathfrak p_n^{1/N}\right\}.
 \label{eq:prokhorov-rate}
\end{equation}
The sum on the right may be replaced by
$\max\{4a_n,8L\mathfrak p_n^{1/N}/c_{K,N,L}\}$.
\end{theorem}

\begin{proof}
Normalize the metric by $L$ and set $\alpha_n=a_n/L$ and
$c=c_{K,N,L}$. Theorem~\ref{thm:ball-mass} gives
\[
 b(s):=\inf_{z\in M}m(B_{sL}(z))\geq(cs)^N,
 \qquad 0<s\leq1.
\]
The limit is connected by Lemma~\ref{lem:geodesic-limit}.
Put
\[
 T_n=\max\{4\alpha_n,8\mathfrak p_n^{1/N}/c\}.
\]
If $T_n\geq1$, the bound $h_n\leq L$ follows from the tail diameter.
Otherwise, choose $T_n<\eps\leq1$.
Then $\alpha_n<\eps/4$ and
\[
 \mathfrak p_n<(c\eps/8)^N\leq b(\eps/8).
\]
Since $N\geq1$, $c\leq1$ and $0\leq\mathfrak p_n\leq1$, also
\[
 \mathfrak p_n\leq\mathfrak p_n^{1/N}<c\eps/8<\eps/4.
\]
Lemma~\ref{lem:upgrade}, applied in $Z_{\ge R}$, gives
$h_n/L\leq\eps$. Letting $\eps\downarrow T_n$ proves the bound
with the maximum, including when $T_n=0$.
The displayed estimate follows by replacing the maximum by the sum.
\end{proof}

\subsection{Geometric resolution and mass accuracy}

The local geometric scale $\ell_n=\max_xs_n(x)$ was fixed in
Section~\ref{sec:refinement}. To estimate approximation to the limit,
we introduce a separate radius $\rho_n$ through a measurable partition
of $M$. The radius controls the distance from each cell to its layer
element; the cell masses are then compared with the reference masses.

\begin{theorem}[A rate from resolution and mass discrepancy]
\label{thm:cell-rate}
Suppose that, for each $n$, there are a radius $\rho_n\geq0$
and a measurable partition $\{A_{n,x}:x\in V_n\}$ of $M$,
up to null sets, such that
\begin{equation}
 w_{n,x}:=m(A_{n,x})>0,\qquad
 d_Z(x,z)\leq\rho_n\quad\text{for $m$-almost every }z\in A_{n,x}.
 \label{eq:cell-certificate}
\end{equation}
Define the mass discrepancy
\begin{equation}
 e_n=\frac12\sum_{x\in V_n}|m_n(x)-w_{n,x}|.
 \label{eq:cell-mass-discrepancy}
\end{equation}
Then
\begin{equation}
 h_n\leq\min\left\{L,\;
 \rho_n+\frac{4L}{c_{K,N,L}}e_n^{1/N}\right\}.
 \label{eq:cell-rate}
\end{equation}
If $m_n(x)=w_{n,x}$ for every $x$, then $h_n\leq\rho_n$.
\end{theorem}

\begin{proof}
If $e_n=0$, then $m_n(x)=w_{n,x}$ for every $x$, and
\[
 \beta_n=\sum_{x\in V_n}\delta_x\otimes(m|_{A_{n,x}})
\]
has first marginal $m_n$, second marginal $m$, and displacement at
most $\rho_n$. This proves the exact-mass case.

Suppose $e_n>0$ and put $r=4Le_n^{1/N}/c_{K,N,L}$.
If $\rho_n+r\geq L$, the conclusion follows from $h_n\leq L$;
assume henceforth that $\rho_n+r<L$.
For nonempty $A\subset V_n$ let $S_A=\bigcup_{x\in A}A_{n,x}$.
The definition of total variation gives
\begin{equation}
 m_n(A)\leq\sum_{x\in A}w_{n,x}+e_n=m(S_A)+e_n.
 \label{eq:cell-hall-start}
\end{equation}
We verify Hall's inequality by finding enough mass in the
$(\rho_n+r)$-neighborhood of $A$ to absorb this discrepancy.
Here $A^s=\{z\in Z:d_Z(z,A)<s\}$.
If $M\subset A^{\rho_n+r}$, then $m(A^{\rho_n+r})=1$ and the
inequality holds. Otherwise, the continuous function $f(z)=d_Z(z,A)$
on the compact connected space $M$ has maximum at least $\rho_n+r$.
Its minimum is at most $\rho_n$: choose $x\in A$ and a point of
its positive-mass cell where \eqref{eq:cell-certificate} holds.
Connectedness therefore gives $z_0\in M$ with
$f(z_0)=\rho_n+r/2$.

For $z\in B_{r/4}(z_0)$, the $1$-Lipschitz property of $f$ gives
\[
 \rho_n+r/4<f(z)<\rho_n+3r/4<\rho_n+r.
\]
Hence $B_{r/4}(z_0)\subset A^{\rho_n+r}$.
Since $f\leq\rho_n$ almost everywhere on $S_A$, the ball is
disjoint from $S_A$ up to a null set.
Since $r/4<L$, Theorem~\ref{thm:ball-mass} applies and gives
\[
 m(B_{r/4}(z_0))\geq
 \left(c_{K,N,L}\frac r{4L}\right)^N=e_n.
\]
Using also $S_A\subset A^{\rho_n+r}$ up to a null set, we obtain
\[
 m_n(A)\leq m(S_A)+e_n
 \leq m(S_A)+m(B_{r/4}(z_0))
 \leq m(A^{\rho_n+r}).
\]
For the empty set, both sides of Hall's inequality are zero.
The same inequality holds with closed neighborhoods, so
Lemma~\ref{lem:hall} gives a coupling of radius at most $\rho_n+r$.
\end{proof}

\begin{remark}[The role of the partition]
The cells need not be nested, and their indices need not persist
between layers. They describe approximation to an identified limit
and are not part of the definition of refinement.
Condition~\eqref{eq:cell-certificate} uses the intrinsic bridge metric.
A bound $\rho_n\leq C_0\ell_n$, when available, must therefore be
verified in this metric.

Section~\ref{sec:limits} gives another route to a coupling estimate.
If reference paths starting at $V_n$ have endpoint displacement at
most $\omega_n\to0$, Theorem~\ref{thm:compact-limit} gives
$h_n\leq\omega_n$.
The local criteria there yield $\omega_n=O(\ell_n)$ under their
respective hypotheses. This estimate can be used directly in the midpoint
recovery theorem, without constructing a partition of $M$.
\end{remark}

\begin{corollary}[Ball masses and packing on a finite layer]
\label{cor:finite-geometry}
Put $h=h_n$. For $x\in V_n$ and $2h<\rho\leq L+2h$,
\begin{equation}
 m_n(B_\rho^{V_n}(x))\geq
 \left(c_{K,N,L}\frac{\rho-2h}{L}\right)^N.
 \label{eq:finite-ball-mass}
\end{equation}
If $2h<r\leq L$, every $r$-separated subset of $V_n$ has cardinality
at most
\begin{equation}
 \left[\frac{2L}{c_{K,N,L}(r-2h)}\right]^N.
 \label{eq:finite-packing}
\end{equation}
\end{corollary}

\begin{proof}
Choose a coupling $\beta_n$ supported within distance $h$.
Since $m_n(x)>0$, its conditional measure over $x$ is a probability.
Choose $z\in M$ in the support of this conditional measure; then
$d_Z(x,z)\leq h$. For every pair $(y,w)$ in the support of $\beta_n$
with $w\in B_{\rho-2h}(z)$, the triangle inequality gives
\[
 d_n(x,y)\leq d_Z(x,z)+d(z,w)+d_Z(w,y)
 <h+(\rho-2h)+h=\rho.
\]
Taking marginals gives
$m_n(B_\rho^{V_n}(x))\geq m(B_{\rho-2h}(z))$.
Since $0<\rho-2h\leq L$, Theorem~\ref{thm:ball-mass} proves
\eqref{eq:finite-ball-mass}.

For $r$-separated points $x_1,\ldots,x_q$, choose corresponding
$z_1,\ldots,z_q\in M$ as above. For $i\ne j$,
\[
 d(z_i,z_j)\geq d_n(x_i,x_j)-d_Z(x_i,z_i)-d_Z(x_j,z_j)
 \geq r-2h.
\]
Their open balls of radius $(r-2h)/2$ are therefore pairwise disjoint.
Theorem~\ref{thm:ball-mass} gives
\[
 1\geq\sum_{i=1}^q m(B_{(r-2h)/2}(z_i))
 \geq q\left(c_{K,N,L}\frac{r-2h}{2L}\right)^N,
\]
which is \eqref{eq:finite-packing}.
\end{proof}

\begin{example}[Uniform midpoint grids]
\label{ex:midpoint-grid-scale}
For $n\geq1$, put $J_n=2^n$ and let $I_{n,j}$ be the dyadic
intervals of length $1/J_n$, with boundary points assigned to form
a partition of $[0,1]$. Set
\[
 x_{n,j}=\frac{j+1/2}{J_n},\qquad
 V_n=\{x_{n,j}:0\leq j<J_n\},\qquad
 m_n(x_{n,j})=\frac1{J_n}.
\]
Give each layer the shortest-path metric obtained by joining
consecutive points with their Euclidean lengths, and use
$b_n(x,y)=|x-y|$ between consecutive layers.
Interval bisection, with each mass divided equally between the two
halves, gives a stationary local rule with
\[
 \ell_n=\frac1{J_n},\qquad
 \ell_{n+1}=\frac{\ell_n}{2}.
\]
Every supported reference transition has displacement $\ell_n/4$,
so \eqref{eq:reference-displacement} holds.

For $t\in[0,1]$, the midpoints of nested dyadic intervals containing
$t$ converge in $Z$ to a point $\iota(t)$: their successive
distances are $2^{-k-2}$.
Boundary choices give the same limit, since their layer-$k$
midpoints are at distance at most $2^{-k}$.
The Euclidean layer metrics show that
$d_Z(\iota(s),\iota(t))=|s-t|$.
The same paths give
\begin{equation}
 d_Z(x_{n,j},\iota(t))
 \leq\sum_{k=n}^{\infty}2^{-k-2}
 =\frac1{2J_n}=:\rho_n
 \qquad(t\in I_{n,j}).
 \label{eq:dyadic-canonical-cell}
\end{equation}
Identify $M=\iota([0,1])$ with the interval, and let $m$ be Lebesgue
probability. The cell coupling
$\sum_j\delta_{x_{n,j}}\otimes(m|_{I_{n,j}})$ has radius at most
$\rho_n$, while every coupling has radius at least $\rho_n$:
for $r<\rho_n$, the interval $[0,\rho_n-r)$ has positive mass
and distance greater than $r$ from every layer point.
Here we use that $d_Z$ dominates coordinate distance.
Consequently,
\[
 e_n=0,\qquad h_n=\rho_n=\frac{\ell_n}{2}.
\]
The cell estimate and coupling give \eqref{eq:convergence}, so
$Z$ is compact by Lemma~\ref{lem:compact}.
The limit interval satisfies $\CDe(0,N)$ for every $N\geq1$.
\end{example}

\begin{remark}[The role of local reference transport]
\label{ex:slow-mass-rate}
Keep the geometry of Example~\ref{ex:midpoint-grid-scale}, but
temporarily omit \eqref{eq:reference-displacement} and replace the
reference masses by
\[
 \widetilde m_n(x_{n,j})
 =\int_{I_{n,j}}[1+\eta_n(2t-1)]\,\dd t,
 \qquad 0<\eta_n < 1, \quad \eta_n \to 0.
\]
Their discrepancy from the exact cell masses is
$\widetilde e_n=\eta_n/4$.
Thus $\widetilde m_n\rightharpoonup m$, and
Lemma~\ref{lem:upgrade} gives
$\widetilde h_n:=W_{\infty,Z}(\widetilde m_n,m)\to0$.
However,
\begin{equation}
 \widetilde h_n\geq\rho_n+\frac{\eta_n}{4}.
 \label{eq:slow-mass-rate}
\end{equation}
Indeed, the left half of the layer has mass $1/2-\eta_n/4$,
whereas every point in the right half has coordinate at least
$1/2+\rho_n$. A coupling of radius $r<\rho_n+\eta_n/4$ would have
to pair all limit mass in $[0,1/2+\rho_n-r)$ with the left half,
exceeding its available mass.

Choosing $\eta_n/\ell_n\to\infty$ gives
$\widetilde h_n/\ell_n\to\infty$.
These masses cannot admit reference couplings satisfying
\eqref{eq:reference-displacement}, since the dyadic contraction
would then give $\widetilde h_n=O(\ell_n)$ by
Proposition~\ref{prop:local-control} and
Theorem~\ref{thm:compact-limit}.
Thus, for these grids, local reference transport also controls
the rate of mass approximation.
\end{remark}

\subsection{Relative comparisons above the coupling scale}
\label{sec:comparison-scales}

The absolute midpoint defects vanish as $h_n\to0$.
To recover the curvature term at shrinking transport scales, we
compare these defects with the squared endpoint distance.
We use $U_n$ from \eqref{eq:finite-entropy} and the midpoint
coefficient $S_\kappa$ from \eqref{eq:mid-coefficient}.
The metric and entropy estimates below hold for the same approximate
midpoint supplied by Theorem~\ref{thm:single-recovery}.

\begin{corollary}[Relative errors at a prescribed comparison scale]
\label{cor:relative-errors}
Let $p,q\in\Pp(V_n)$ with $D=W_n(p,q)>0$. There is
$\nu\in\Pp(V_n)$ such that, with
$a=W_n(p,\nu)$ and $b=W_n(\nu,q)$,
\begin{align}
0\le \frac{2a^2+2b^2-D^2}{D^2}
&\le 12\frac{h_n}{D}+36\left(\frac{h_n}{D}\right)^2,
\label{eq:relative-action}\\
\frac{\left[S_\kappa(D)-
 U_n(\nu)/(U_n(p)+U_n(q))\right]_+}{D^2}
&\le 2C_{\kappa,L}
 \left[\frac{h_n}{D}+\left(\frac{h_n}{D}\right)^2\right],
\label{eq:relative-entropy}
\end{align}
where $[s]_+=\max\{s,0\}$ and $C_{\kappa,L}$ is the derivative
constant in Theorem~\ref{thm:single-recovery}.
Consequently, if $r_n>0$ and $h_n/r_n\to0$, both right-hand sides
tend to zero uniformly over all endpoints with $D\ge r_n$.
\end{corollary}

\begin{proof}
Choose $\nu$ as in Theorem~\ref{thm:single-recovery}.
The triangle inequality and $2(a^2+b^2)\geq(a+b)^2$ give
\[
 2a^2+2b^2-D^2\geq(a+b)^2-D^2\geq0.
\]
Divide the bound $2a^2+2b^2-D^2\leq12Dh_n+36h_n^2$ in
\eqref{eq:sharp-action} by $D^2$ to obtain
\eqref{eq:relative-action}.
Every probability on $V_n$ has finite entropy because the reference
masses are positive. Thus $U_n(p)+U_n(q)>0$, and division of
\eqref{eq:sharp-entropy} by this sum yields
\[
 \frac{U_n(\nu)}{U_n(p)+U_n(q)}
 \geq S_\kappa(D)-2C_{\kappa,L}h_n(D+h_n).
\]
Taking the positive part of the deficit and dividing by $D^2$ proves
\eqref{eq:relative-entropy}. For $D\geq r_n>0$, both right-hand
sides are bounded in terms of
$h_n/r_n$, which proves the uniform convergence.
\end{proof}

\begin{remark}[Concentrated or coincident endpoints]
\label{rem:concentration}
No uniform positive lower bound on $U_n(p)+U_n(q)$ is required in
Corollary~\ref{cor:relative-errors}. In particular, its estimates include
Dirac measures and other endpoint families whose entropies diverge with
$n$. This follows from the multiplicative form of
\eqref{eq:sharp-entropy}, which retains the factor $U_n(p)+U_n(q)$.
When $D=0$, take $\nu=p=q$. Both midpoint comparisons then hold exactly,
with $S_\kappa(0)=1/2$, and no normalization by $D$ is used.
\end{remark}

\begin{corollary}[The curvature scale of the entropy comparison]
\label{cor:curvature-scale}
For fixed $\kappa$,
\begin{equation}
S_\kappa(D)=\frac12+\frac{\kappa}{16}D^2+O_\kappa(D^4)
\qquad(D\to0).
\label{eq:midpoint-expansion}
\end{equation}
Suppose $0<r_n\le R_n$, $R_n\to0$, and $h_n/r_n\to0$.
For every endpoint pair with $r_n\le D\le R_n$, the approximate
midpoint $\nu$ supplied by Theorem~\ref{thm:single-recovery} satisfies
\begin{equation}
\frac{U_n(\nu)/(U_n(p)+U_n(q))-1/2}{D^2}
\ge \frac{K}{16N}
 -2C_{\kappa,L}
 \left[\frac{h_n}{r_n}+\left(\frac{h_n}{r_n}\right)^2\right]
 -A_\kappa R_n^2
\label{eq:curvature-scale-bound}
\end{equation}
for all sufficiently large $n$, with a constant $A_\kappa$ independent
of the endpoints. Its relative action error tends to zero uniformly
on the same class.
\end{corollary}

\begin{proof}
The half-angle formulas in \eqref{eq:mid-coefficient}, together with
$\cos x=1-x^2/2+O(x^4)$ and $\cosh x=1+x^2/2+O(x^4)$,
give \eqref{eq:midpoint-expansion}. In particular, for fixed $\kappa$
there exist $D_0>0$ and $A_\kappa\geq0$ such that
\[
 \left|S_\kappa(D)-\frac12-\frac{\kappa}{16}D^2\right|
 \leq A_\kappa D^4\qquad(0\leq D\leq D_0).
\]
For $\kappa=0$ the expansion is exact and $A_\kappa=0$ is admissible.
For all sufficiently large $n$, $R_n\leq D_0$.
Choose $\nu$ as in Theorem~\ref{thm:single-recovery} and apply
\eqref{eq:sharp-entropy}. After dividing by $U_n(p)+U_n(q)$,
subtracting $1/2$, and dividing by $D^2$, we obtain
\[
 \frac{U_n(\nu)/(U_n(p)+U_n(q))-1/2}{D^2}
 \geq\frac{\kappa}{16}-A_\kappa D^2
 -2C_{\kappa,L}\left[\frac{h_n}{D}
                         +\left(\frac{h_n}{D}\right)^2\right].
\]
Now use $0<r_n\leq D\leq R_n$ and $\kappa=K/N$ to obtain
\eqref{eq:curvature-scale-bound}. Corollary~\ref{cor:relative-errors}
gives the uniform relative action estimate for the same $\nu$.
\end{proof}

The condition $h_n=o(r_n)$ suffices because $S_\kappa'(0)=0$.
This gives a coefficient error of order $h_n(D+h_n)$; a uniform
Lipschitz bound would only give an error of order $h_n$.
The resulting lower comparison in \eqref{eq:curvature-scale-bound}
holds at a fixed dimension parameter $N$. It does not identify the
optimal curvature bound or the smallest admissible dimension from a
single endpoint pair or a single choice of approximate midpoint.

\begin{corollary}[Sufficient resolution and mass accuracy]
\label{cor:resolution-regime}
Under the assumptions of Theorem~\ref{thm:cell-rate}, we have
\begin{equation}
 h_n\leq\rho_n+\frac{4L}{c_{K,N,L}}e_n^{1/N}.
 \label{eq:scale-input-rate}
\end{equation}
For comparison scales $r_n>0$, the condition
\begin{equation}
 \rho_n+Le_n^{1/N}=o(r_n)
 \label{eq:resolution-test-separation}
\end{equation}
is sufficient for the uniform estimates of
Corollary~\ref{cor:relative-errors}.
If also $0<r_n\leq R_n\to0$, the conclusion of
Corollary~\ref{cor:curvature-scale} holds.
In particular, if
\begin{equation}
 e_n=O\!\left((\rho_n/L)^N\right),
 \label{eq:mass-mesh-balance}
\end{equation}
then $h_n=O(\rho_n)$, and $\rho_n=o(r_n)$ suffices.
If, in addition, $\rho_n=O(\ell_n)$, it suffices to require
$\ell_n=o(r_n)$.
\end{corollary}

\begin{proof}
The constants $L$ and $c_{K,N,L}>0$ are fixed.
Dividing \eqref{eq:scale-input-rate} by $r_n$ shows that
\eqref{eq:resolution-test-separation} implies $h_n/r_n\to0$.
The two preceding corollaries apply.
Under \eqref{eq:mass-mesh-balance},
$Le_n^{1/N}=O(\rho_n)$, so $h_n=O(\rho_n)$.
The remaining assertions follow by substitution.
\end{proof}

The same comparison scales can be obtained from the reference
transport. Under any of the local endpoint-control criteria in
Section~\ref{sec:limits}, we have $h_n\leq C_*\ell_n$ for all
sufficiently large $n$, with a constant $C_*$ independent of $n$.
Thus $\ell_n=o(r_n)$ again suffices.
For example, Proposition~\ref{prop:local-control} gives
\[
 h_n\leq\frac{C\ell_n}{1-q}
 \qquad\text{if }\ell_{n+1}\leq q\ell_n,\quad 0<q<1,
\]
where $C$ is the constant in \eqref{eq:reference-displacement}.
The partition estimate and the reference-path estimate therefore
provide two sufficient ways to control the same relative errors.

\subsection{The nearest-neighbor obstruction}

The preceding estimates apply above the coupling scale. At the
nearest-neighbor scale, a finite layer may retain a positive relative
action defect. The obstruction is already visible for Dirac endpoints.
\begin{corollary}[An obstruction at the nearest-neighbor scale]
\label{prop:nearest-scale}
Let $|V_n|\geq2$, put $\delta_n=\min_{x\ne y}d_n(x,y)$, and choose
$x,y$ with $d_n(x,y)=\delta_n$.
For $p=\delta_x$ and $q=\delta_y$,
\begin{equation}
 \min_{\nu\in\Pp(V_n)}
 \bigl[2W_n(p,\nu)^2+2W_n(\nu,q)^2-W_n(p,q)^2\bigr]
 =\delta_n^2.
 \label{eq:nearest-action-floor}
\end{equation}
Consequently, the minimum relative action defect for this pair is $1$.
\end{corollary}

\begin{proof}
The Dirac-pair estimate in Proposition~\ref{prop:finite-obstruction}
gives the lower bound $\delta_n^2$.
For $\nu=\tfrac12(\delta_x+\delta_y)$, both squared distances to the
endpoints equal $\delta_n^2/2$, so equality is attained.
\end{proof}

For the uniform grids of Example~\ref{ex:midpoint-grid-scale},
\[
 \delta_n=\ell_n=2h_n.
\]
Corollary~\ref{prop:nearest-scale} therefore gives a minimum action
defect $\delta_n^2=4h_n^2\to0$ for a closest Dirac pair, while its
minimum relative defect is $1$ at every layer.
This is consistent with the fixed normalization $L^2$ used in
Section~\ref{sec:characterization}.
It also shows that the uniform vanishing of relative action errors
cannot in general be extended to transport distances comparable
to the coupling error $h_n$.
\section{Parameter regions, invariance, and finite defects}
\label{sec:invariance}

The preceding sections characterize $\SCD(K,N)$ by the geometry of
the measured limit and estimate the corresponding errors on finite layers.
We now study parameter monotonicity and invariance under deletion
or grouping of refinement steps. We also quantify the simultaneous
admissibility and entropy errors on finite segments.

We retain the standing assumptions of Section~\ref{sec:entropy}.
Convergence is assumed only where stated.

\subsection{Parameter monotonicity}

Within the comparison range, define the admissible parameter region by
\[
 \mathcal C(\Sigma)
 =\{(K,N)\in\mathbb R\times[1,\infty):
       \Sigma\text{ satisfies }\SCD(K,N)\}.
\]

\begin{proposition}[Monotonicity]
\label{prop:parameter-monotonicity}
If $\Sigma\text{ satisfies }\SCD(K,N)$, 
then
$\Sigma\text{ satisfies }\SCD(K',N')$ 
whenever $K'\le K$ and $N\le N'<\infty$.
\end{proposition}
\begin{proof}
The admissible classes depend only on the metric geometry and the
tolerance. Increasing $N$ changes the entropy powers by a fractional
power, so it remains to compare the entropy bounds.
Fix a segment of length $J$ and endpoint probabilities. Write
$a,b>0$ for their entropy powers at dimension $N$ and $D$ for their
Wasserstein distance, and set
\[
 \phi(t)=\sigma_{K/N}^{(1-t)}(D)a+\sigma_{K/N}^{(t)}(D)b,
 \qquad q=N/N'\in(0,1].
\]
The positive function $\phi$ solves $\phi''+(K/N)D^2\phi=0$.
Direct differentiation gives
\[
 (\phi^q)''+\frac K{N'}D^2\phi^q
   =q(q-1)\phi^{q-2}(\phi')^2\le0.
\]
Since $K'\le K$, we also have
$(\phi^q)''+(K'/N')D^2\phi^q\le0$. Scalar Dirichlet comparison
therefore yields
\[
 \phi(t)^q\ge
 \sigma_{K'/N'}^{(1-t)}(D)a^q+
 \sigma_{K'/N'}^{(t)}(D)b^q.
\]
Indeed, the difference $w$ between the two sides has zero endpoint
values and satisfies $w''+cw\le0$, where
$c=(K'/N')D^2<\pi^2$. 
Testing against the negative part $w_-=\max\{-w,0\}$ gives
\[
 \int_0^1|(w_-)'|^2\le c\int_0^1w_-^2.
\]
The Dirichlet Poincar\'e inequality implies $w_-=0$.
The bound on $c$ follows from the comparison range and
$K'\le K$, $N'\ge N$.

Apply this inequality to a section
$\boldsymbol\mu\in\mathcal G_{r,J}^{\delta}(\mu_0,\mu_J)$
satisfying the entropy comparisons with tolerance $\delta$.
Since $U_{N'}(\mu\mid m)=U_N(\mu\mid m)^q$ and
$|x^q-y^q|\le|x-y|^q$ for $x,y\ge0$, we obtain
\[
 U_{N'}(\mu_j\mid m_{r+j})
 \ge(\phi(j/J)-\delta)_+^q
 \ge\phi(j/J)^q-\delta^q.
\]
Given $\eps>0$, choose $\delta>0$ with
$\delta\le\eps$ and $\delta^q\le\eps$.
The same section belongs to
$\mathcal G_{r,J}^{\eps}(\mu_0,\mu_J)$ and satisfies the entropy
comparisons at $(K',N')$ with tolerance $\eps$.
This choice is independent of the endpoints, proving
$\SCD(K',N')$.
\end{proof}

\subsection{Tail invariance and grouped refinement steps}

Since the $\SCD(K,N)$ condition concerns sufficiently late segments
of each fixed length, deleting finitely many initial layers
preserves the condition.
To be more precise, by Lemma~\ref{lem:gluing}, 
deleting the initial layers leaves the intrinsic bridge metric unchanged 
on all retained layers, 
and hence preserves the admissible sections and entropy comparisons there. 

Furthermore, if two complete refinement systems satisfying
the standing assumptions converge to the same measured limit,
Theorem~\ref{thm:characterization} shows that their
$\SCD(K,N)$ conditions are equivalent within a common
comparison range.
These observations give the following proposition directly.

\begin{proposition}[Tail and limit invariance]
\label{cor:tail-invariance}
Let $\Sigma$ and $\widetilde\Sigma$ be complete refinement systems 
satisfying the standing assumptions. Then we have
\begin{enumerate}
 \item Deleting finitely many initial layers of a refinement system $\Sigma$
          preserves its $\SCD(K,N)$ condition.
\item If both systems converge in the sense of
         Definition~\ref{def:convergence}
         and their limits are related by a measure-preserving isometry,
         then their $\SCD(K,N)$ conditions agree in every common
         comparison range.
\end{enumerate}
\end{proposition}

For a subsequence, the bridges between selected layers must also
be specified. We use the distances obtained by gluing the original
bridges across the omitted layers.

\begin{proposition}[Subsequences and grouped steps]
\label{prop:schedule-invariance}
Suppose $\Sigma$ converges in the sense of
Definition~\ref{def:convergence}, and let $n_0<n_1<\cdots$ tend to
infinity. Between $V_{n_i}$ and $V_{n_{i+1}}$, take the restriction of
$D_{[n_i,n_{i+1}]}$ to these two layers as the new bridge, and compose
the reference kernels over the omitted layers. The resulting
comparison system $\widetilde\Sigma$ has the same measured limit.
Within a common comparison range, the admissibility and entropy
comparisons of Definition~\ref{def:scd} hold for the selected layers
if and only if $\Sigma$ satisfies $\SCD(K,N)$.
\end{proposition}
\begin{proof}
We first verify the reference marginals and the intrinsic bridge
metric. Write $\widetilde V_i=V_{n_i}$ and set
\[
 \widetilde P_i=P_{n_i}\cdots P_{n_{i+1}-1},
 \qquad
 \widetilde Q_i=Q_{n_{i+1}-1}\cdots Q_{n_i}.
\]
Each product consists of a single kernel when $n_{i+1}=n_i+1$.
The reference marginal identities give
\[
 m_{n_i}\widetilde P_i=m_{n_{i+1}},\qquad
 m_{n_{i+1}}\widetilde Q_i=m_{n_i}.
\]
The defining relation between $P_k$ and $Q_k$ also yields
\[
 \widetilde\gamma_i(x,y)
 :=m_{n_i}(x)\widetilde P_i(x,y)
 =m_{n_{i+1}}(y)\widetilde Q_i(y,x).
\]
Thus $\widetilde\gamma_i$ is a reference coupling with the required
marginals.

Let $\widetilde d$ be the intrinsic bridge metric obtained from the
grouped bridges. By Lemma~\ref{lem:gluing}, each grouped bridge is
a restriction of $d_Z$. Every chain through the grouped bridges
has length at least the $d_Z$-distance between its endpoints, so
$d_Z(x,y)\le\widetilde d(x,y)$.

For the reverse inequality, take $x\in V_{n_i}$ and
$y\in V_{n_j}$ with $i<j$. Formula~\eqref{eq:monotone-chain}
gives a shortest chain $x_k\in V_k$, $n_i\le k\le n_j$,
with $x_{n_i}=x$ and $x_{n_j}=y$. Split this chain at the selected
indices. On each block, the grouped bridge distance is at most the
sum of the original bridge lengths. Hence
\begin{align*}
 \widetilde d(x,y)
 \le\sum_{\ell=i}^{j-1}
      D_{[n_\ell,n_{\ell+1}]}(x_{n_\ell},x_{n_{\ell+1}})
 \le\sum_{k=n_i}^{n_j-1}b_k(x_k,x_{k+1})
 = d_{Z}(x,y).
\end{align*}
On a single selected layer, both metrics equal its given metric.
They therefore agree on all selected points, including the
identifications at zero distance.

The new completion is consequently isometric to
$\overline{\bigcup_iV_{n_i}}\subset Z$. It contains $M$, since
$d_H^Z(V_{n_i},M)=a_{n_i}\to0$. The original Hausdorff and weak
convergence restrict to this closed set, giving
\[
 d_H(V_{n_i},M)\longrightarrow0,
 \qquad m_{n_i}\rightharpoonup m.
\]
Choose $i_0$ so that $n_{i_0}\ge R$. The selected tail from $i_0$
onward inherits the same diameter bound $L$ and comparison range.

It remains to compare the entropy conditions. By
\eqref{eq:winfty-convergence}, the original system admits couplings
to $m$ with radii $h_n=W_{\infty,Z}(m_n,m)\to0$.
The selected layers inherit these couplings with radii $h_{n_i}$.
If $\Sigma$ satisfies $\SCD(K,N)$, then $M$ satisfies $\CDe(K,N)$
by Theorem~\ref{thm:characterization}. Lift the selected endpoint
measures to $M$ and project one comparison geodesic to the selected
intermediate layers. The proof of
Proposition~\ref{prop:window-recovery} then gives the required
admissibility and entropy comparisons. Its errors tend uniformly
to zero because $\sup_{k\ge i}h_{n_k}\to0$.

Conversely, suppose the selected layers satisfy those comparisons.
Use the case $J=2$ and recover fixed finite-entropy endpoints on
$M$ at the selected layers with indices $i$ and $i+2$.
Lemma~\ref{lem:recovery} gives convergence of both endpoint
measures and their entropy values. Choose a sequence of starting
indices along which the comparison tolerances tend to zero.
Compactness and joint upper semicontinuity of the entropy power
then give an exact entropy comparison midpoint on $M$, by the
limit argument in the proof of Theorem~\ref{thm:characterization}.
Lemma~\ref{lem:dyadic} yields $\CDe(K,N)$ on $M$, and
Theorem~\ref{thm:characterization} gives $\SCD(K,N)$ for $\Sigma$.
These two arguments use the inherited metric, convergence, and
reference couplings; they require no generating rule for the
selected layers.
\end{proof}

\begin{remark}
The proposition concerns the comparison data on the selected layers.
Grouping need not preserve the original stationary local rule or
the displacement estimate \eqref{eq:reference-displacement}
with the same constant.
The composed reference coupling need not be optimal
for the grouped cost either.
\end{remark}

\subsection{Finite comparison defects}

The admissible parameter region does not determine the comparison
errors on a given finite segment. We record the least common
tolerance for admissibility and all entropy comparisons, uniformly
over the endpoint measures.

Fix $r\ge R$ and $J\ge2$. For a probability section
$\boldsymbol\mu=(\mu_0,\ldots,\mu_J)$, use the action
$\mathcal A_{r,J}(\boldsymbol\mu)$ and endpoint distance
$D=W(\mu_0,\mu_J)$ from \eqref{eq:section-action}, and put
\[
 \Phi_j=\sigma_{K/N}^{(1-j/J)}(D)U_r(\mu_0)
       +\sigma_{K/N}^{(j/J)}(D)U_{r+J}(\mu_J).
\]
Here $U_{r+j}(\mu_j)=U_N(\mu_j\mid m_{r+j})$. Define
\begin{equation}
 \delta_{r,J}^{K,N}
 =\sup_{\mu_0,\mu_J}\ \inf_{\mu_1,\ldots,\mu_{J-1}}
 \max\left\{
 \frac{\mathcal A_{r,J}(\boldsymbol\mu)-D^2}{L^2},
 \max_{1\le j<J}(\Phi_j-U_{r+j}(\mu_j))_+
 \right\}.
 \label{eq:finite-defect}
\end{equation}
The supremum is over $\mu_0\in\Pp(V_r)$ and
$\mu_J\in\Pp(V_{r+J})$; the infimum is over all
$\mu_j\in\Pp(V_{r+j})$, $1\le j<J$.
The action excess is nonnegative by \eqref{eq:action-lower-bound}.

Both extrema are attained. To see this, write
$e=(\mu_0,\mu_J)$ for the endpoints and $\boldsymbol\nu$ for the
interior measures. The quantity being minimized is a continuous
function $F(e,\boldsymbol\nu)$ on a compact product of finite
probability simplexes. Continuity of the entropy powers follows
from the full support of the reference measures. Thus
$g(e)=\min_{\boldsymbol\nu}F(e,\boldsymbol\nu)$ exists and satisfies
\[
 |g(e)-g(e')|
 \le\sup_{\boldsymbol\nu}
       |F(e,\boldsymbol\nu)-F(e',\boldsymbol\nu)|.
\]
Uniform continuity of $F$ makes $g$ continuous, so it attains a
maximum on the endpoint space.
Consequently, for $\eps>0$, the bound
$\delta_{r,J}^{K,N}\le\eps$ is equivalent to the existence, for
every endpoint pair, of a section in
$\mathcal G_{r,J}^{\eps}(\mu_0,\mu_J)$ satisfying all entropy
comparisons with tolerance $\eps$.

For $J=2$, this defect compares measures on three consecutive
layers. The defect $\Delta_n(K,N)$ of \eqref{eq:single-defect}
instead places all three measures on one layer. Their finite
values need not agree.

\begin{proposition}[Defect characterization]
\label{prop:defect-characterization}
Under the standing assumptions, 
we have the following equivalence where 
$r \to \infty$, 
\[
 \Sigma\text{ satisfies }\SCD(K,N)
 \quad\Longleftrightarrow\quad
 \delta_{r,J}^{K,N}\to0\ \text{for every fixed }J\ge2
 \quad\Longleftrightarrow\quad
 \delta_{r,2}^{K,N}\to0.
\]
Furthermore, suppose in addition that \eqref{eq:convergence} holds and the limit
satisfies $\CDe(K,N)$. Retain $h_n$ from
\eqref{eq:winfty-convergence} and set
\[
 \widehat h_r=\sup_{n\ge r}h_n,\qquad
 \eta_r=\widehat h_r/L.
\]
Let $B_{\kappa,L}$ be the derivative bound in
Proposition~\ref{prop:window-recovery}, with $\kappa=K/N$.
Then, for every $r\ge R$ and integer $J\ge2$,
\begin{equation}
 \delta_{r,J}^{K,N}\le
 \max\{4(J+1)\eta_r+4J^2\eta_r^2,\;
          4LB_{\kappa,L}\eta_r\}.
 \label{eq:defect-rate}
\end{equation}
\end{proposition}
\begin{proof}
The interpretation of \eqref{eq:finite-defect} shows that vanishing
of the defects for every fixed $J$ is precisely
Definition~\ref{def:scd}. For $J=2$, it is the midpoint condition.
Theorem~\ref{thm:direct-midpoint} gives the remaining equivalence.
Its admissible sections are supplied by
Theorem~\ref{thm:admissible-existence} under the standing complete
refinement assumptions.

Under convergence, Lemma~\ref{lem:upgrade} gives $h_n\to0$, and
hence $\widehat h_r\to0$. When the limit satisfies $\CDe(K,N)$,
Proposition~\ref{prop:window-recovery} supplies, for each endpoint
pair, one section with action excess at most
\[
 4(J+1)L\widehat h_r+4J^2\widehat h_r^2
\]
and entropy deficit at most $4B_{\kappa,L}\widehat h_r$ at every
intermediate layer. Divide the action bound by $L^2$ and use
$\widehat h_r=L\eta_r$. Taking the minimum over sections and then
the maximum over endpoints proves \eqref{eq:defect-rate}.
\end{proof}

Under convergence, Theorem~\ref{thm:characterization} also makes
these conditions equivalent to $\Delta_n(K,N)\to0$ and to
$\CDe(K,N)$ on the limit. 
Thus, within the chosen comparison range, 
the measured limit determines $\mathcal C(\Sigma)$, 
while the finite defects and their rates retain 
information about the chosen refinement sequence.

When the limit satisfies $\CDe(K,N)$, the estimates of
Section~\ref{sec:quantitative-recovery} give explicit bounds for
\eqref{eq:defect-rate}. Under the partition
hypotheses of Theorem~\ref{thm:cell-rate},
\[
 \widehat h_r\le\sup_{n\ge r}
 \left(\rho_n+\frac{4L}{c_{K,N,L}}e_n^{1/N}\right),
\]
where $\rho_n$ is the cell radius and $e_n$ is the reference mass
discrepancy. Alternatively, an endpoint-control estimate
$h_n\le C_*\ell_n$ for all sufficiently large $n$ gives
$\widehat h_r\le C_*\ell_r$ for all sufficiently large $r$, since
$\ell_n$ is nonincreasing.
Either bound can be substituted into \eqref{eq:defect-rate}.
The estimate also permits integer segment lengths $J=J(r)\ge2$ to grow:
$\delta_{r,J(r)}^{K,N}\to0$ whenever $J(r)\eta_r\to0$.
This is a sufficient quantitative condition beyond the fixed-length
comparisons in the definition of SCD.

The defects above use the fixed normalization $L^2$ for the action.
For the single-layer midpoint tests,
Corollary~\ref{cor:relative-errors} instead compares the action
excess with $D^2$. Its relative action and entropy errors vanish
uniformly for $D\ge r_n$ whenever $h_n=o(r_n)$.
The endpoint measures may vary with $n$; the scale $r_n$ specifies
which pairs are included. Corollary~\ref{prop:nearest-scale}
shows that absolute defects can vanish while the relative action
defect stays positive at the nearest-neighbor scale on a finite layer.

\begin{remark}[Relative comparison errors]
For general interior times of interlayer sections, a further
question is to obtain analogous relative estimates on prescribed
families of endpoint pairs and interior times.
For a probability section, write $t=j/J$, $U_0=U_r(\mu_0)$,
$U_t=U_{r+j}(\mu_j)$, and $U_1=U_{r+J}(\mu_J)$, and define
\[
 E=\bigl[\sigma_{K/N}^{(1-t)}(D)U_0
       +\sigma_{K/N}^{(t)}(D)U_1-U_t\bigr]_+,
 \qquad 0<t<1.
\]
One may seek a choice of one probability section for each
prescribed endpoint pair such that
\begin{equation}
 \sup_{D>0}
 \frac{E}{t(1-t)D^2U_*}\longrightarrow0,
 \qquad U_*=\max\{U_0,U_t,U_1\},
 \label{eq:relative-error-target}
\end{equation}
as $r\to\infty$, together with uniformly vanishing relative
action defects for the same sections.
The prescribed families must specify the endpoint scales and
interior times.
Such estimates would quantify the entropy comparison at the
transport scale, beyond the absolute defects considered here.
\end{remark}

\section{Examples of refinement sequences}
\label{sec:examples}

In this section, we discuss examples that illustrate the roles of 
the refinement rule, the layer metrics, and the reference measures in the SCD condition. 
Adaptive interval refinement gives an exact curvature range and explicit midpoint error bounds. 
Two refinements of the same initial hypercube graph 
show how the choice of layer metrics affects the curvature--dimension condition. 
We then present two examples separating endpoint control, 
complete resolution, and the entropy comparison. 
Finally, the Goldberg--Coxeter subdivision sequence of $C_{60}$ exhibits 
a local midpoint obstruction to $\mathrm{SCD}(0,N)$ for every finite $N$, 
established directly from the refinement layers 
without assuming the existence of a limiting space.

\subsection{Weighted interval}
\label{sec:adaptive-weighted-interval}

Starting from $V_0=\{0,1\}$, during each refinement step, 
insert finitely many points into selected
intervals between consecutive vertices, retaining all old vertices.
Write
\[
 V_n=\{0=x_{n,0}<\cdots<x_{n,k_n}=1\},\qquad
 \ell_n=\max_{i<k_n}(x_{n,i+1}-x_{n,i}),
\]
and assume $\ell_n\to0$.
The selected intervals and cut positions are marks for the same local rule. 
For example, 
inserting every midpoint at each step gives binary refinement.
Join consecutive vertices by edges of their Euclidean lengths
and use coordinate distances for the bridges:
\[
 d_n(x,y)=|x-y|,\qquad b_n(x,y)=|x-y|.
\]
For these point elements, $r_n(x)=0$ and
$s_n(x)=\max_{z\sim x}|x-z|$.
Since every new edge lies in an old interval,
the local resolution condition \eqref{eq:local-resolution} holds,
with maximal scale $\ell_n$.

Fix $\psi\in C^2([0,1])$ and set
\[
 m=Z_\psi^{-1}e^{-\psi(x)}\,\dd x,\qquad
 Z_\psi=\int_0^1e^{-\psi(x)}\,\dd x.
\]
We construct the reference measures by projecting the same probability
measure $m$ onto each vertex set $V_n$.
For $t\in[0,1]$, let $P_n(t)$ be the nearest vertex in $V_n$,
choosing the smaller vertex in case of a tie.
Each vertex $x$ receives all the mass of the interval
\[
 C_n(x)=\{t\in[0,1]:P_n(t)=x\}.
\]
Thus the discrete reference measure is $m_n=(P_n)_\#m$.

To couple two consecutive layers, we follow the same source mass
under the two projections: the mass at $t$ is assigned to $P_n(t)$
on layer $n$ and to $P_{n+1}(t)$ on layer $n+1$.
Consequently,
\begin{equation}\label{eq:adaptive-weighted-mass}
 m_n(x)=m(C_n(x)),\qquad
 \gamma_n(x,y)=m(C_n(x)\cap C_{n+1}(y)).
\end{equation}

This coupling has marginals $m_n,m_{n+1}$ and is optimal for
squared bridge distance, since both projections are nondecreasing.
Moreover,
\[
 \gamma_n(x,y)>0
 \quad\Longrightarrow\quad
 b_n(x,y)\le\frac{\ell_n+\ell_{n+1}}2\le\ell_n,
\]
which verifies \eqref{eq:reference-displacement}.

Nestedness gives $d_Z(x,y)=|x-y|$, with completion $[0,1]$.
Since $|P_n(x)-x|\le\ell_n/2$ and $(P_n)_\#m=m_n$,
\begin{equation}\label{eq:adaptive-weighted-convergence}
 d_H^Z(V_n,[0,1])\le\frac{\ell_n}{2},\qquad
 W_{\infty,Z}(m_n,m)\le\frac{\ell_n}{2}.
\end{equation}

\begin{proposition}[Weighted interval]
\label{prop:adaptive-weighted-interval}
Let $\Sigma$ be the refinement system constructed above.
For $1<N<\infty$,
set
\begin{equation}\label{eq:adaptive-weighted-curvature}
 K_*(N)=\min_{x\in[0,1]}
 \left\{\psi''(x)-\frac{\psi'(x)^2}{N-1}\right\}.
\end{equation}
For parameters in the standing comparison range,
\begin{equation}\label{eq:adaptive-weighted-equivalence}
 \Sigma\text{ satisfies }\SCD(K,N)\quad\Longleftrightarrow\quad K\le K_*(N).
\end{equation}
For $N=1$, the condition holds exactly when $\psi$ is constant
and $K\le0$. 
\end{proposition}

\begin{proof}
We first prove sufficiency by verifying the midpoint condition
with explicit errors controlled by the mesh size.
For a probability $\xi$ on $V_k$, spread each atom over its control
interval with constant density relative to $m$:
\[
 \widehat\xi=\sum_{x\in V_k}
 \frac{\xi(x)}{m_k(x)}\mathbf1_{C_k(x)}m.
\]
This preserves entropy and gives
$W_2(\xi,\widehat\xi)\le\ell_k/2$.
Conversely, Jensen's inequality gives
$H_k((P_k)_\#\mu)\le\Ent_m(\mu)$.

Assume $N>1$ and $K\le K_*(N)$.
Fix $n$, $p\in\Pp(V_n)$ and $q\in\Pp(V_{n+2})$,
and set $D=W(p,q)$.
Let $Q_-,Q_+$ be the increasing quantiles of $\widehat p,\widehat q$,
and put
\[
 Q_t=(1-t)Q_-+tQ_+,\quad v=Q_+-Q_-,\quad
 A=\int_0^1v^2\,\dd s.
\]
Their positive-density pieces form a common finite partition of
$(0,1)$; jumps merely record gaps in the supports.
The entropy $H(t)$ of $(Q_t)_\#\dd s$ satisfies
\[
 H(t)=\int_0^1[-\log\partial_sQ_t+\psi(Q_t)+\log Z_\psi]\,\dd s.
\]
With $a_t=\partial_sv/\partial_sQ_t$ and $b_t=\psi'(Q_t)v$,
differentiation on these pieces gives
\[
 H'=\int_0^1(b_t-a_t)\,\dd s,\qquad
 H''=\int_0^1[a_t^2+\psi''(Q_t)v^2]\,\dd s.
\]
The identity
\[
 a_t^2+\frac{b_t^2}{N-1}-\frac{(b_t-a_t)^2}{N}
 =\frac{((N-1)a_t+b_t)^2}{N(N-1)}
\]
and Cauchy--Schwarz yield $H''\ge KA+H'^2/N$.
Thus $u=e^{-H/N}$ satisfies $u''+(K/N)Au\le0$, and comparison gives
\begin{equation}\label{eq:adaptive-weighted-auxiliary}
 u(1/2)\ge S_{K/N}(\sqrt A)
 \bigl[U_n(p)+U_{n+2}(q)\bigr].
\end{equation}
For $K>0$, the required bound $(K/N)A<\pi^2$ follows from $A\le1$
and the standing comparison range.

Take $\nu=(P_{n+1}\circ Q_{1/2})_\#\dd s$ and write
\[
 X_-=P_n\circ Q_-,\quad Y=P_{n+1}\circ Q_{1/2},\quad
 X_+=P_{n+2}\circ Q_+.
\]
These functions are nondecreasing, so their common parameter realizes
all three optimal couplings. The projection bounds give
\[
 |\sqrt A-D|\le\ell_n,\qquad
 \left|Y-\frac{X_-+X_+}{2}\right|\le\ell_n.
\]
Put $C_S=\operatorname{Lip}(S_{K/N}|_{[0,1]})$,
which is finite in the standing comparison range.
Jensen's inequality, \eqref{eq:adaptive-weighted-auxiliary},
and the projection bounds give
\begin{equation}\label{eq:adaptive-weighted-midpoint}
 U_{n+1}(\nu)\ge
 S_{K/N}(D)[U_n(p)+U_{n+2}(q)]-2C_S\ell_n.
\end{equation}
Moreover,
\begin{equation}\label{eq:adaptive-weighted-action}
\begin{aligned}
 2W(p,\nu)^2+2W(\nu,q)^2-D^2
 &=4\int_0^1
 \left|Y-\frac{X_-+X_+}{2}\right|^2\,\dd s \le4\ell_n^2.
\end{aligned}
\end{equation}
With $L=1$, the uniform error
$\max\{4\ell_n^2,2C_S\ell_n\}\to0$ verifies $\SCDmid(K,N)$;
Theorem~\ref{thm:direct-midpoint} gives $\SCD(K,N)$.
If $N=1$ and $\psi$ is constant, then
$H''=\int a_t^2\,\dd s\ge H'^2$, giving the same conclusion for $K\le0$.

For sharpness, fix $x_0\in(0,1)$ and an interval $I$ of length $a$
centered at $x_0$. For $\alpha\in\R$, put
\[
 \mu_0=a^{-1}\mathbf1_I\,\dd x,\qquad
 v(x)=1+\alpha(x-x_0),\qquad
 \mu_t=(\mathrm{id}+tv)_\#\mu_0,\qquad
 c^2=\int v^2\,\dd\mu_0.
\]
Choose $\tau>0$ so that these maps are increasing and their images
lie in $(0,1)$ for $|t|\le\tau$.
Assume $\SCDmid(K,N)$. Apply it with tolerances $\eta_i \to 0$
on segments starting at $n_i\to\infty$ to the projected endpoints
$p_i=(P_{n_i})_\#\mu_{-\tau}$ and
$q_i=(P_{n_i+2})_\#\mu_\tau$, obtaining $\nu_i$.
The quantile action identity gives
\[
 W_2(\nu_i,\mu_0)
 \le\frac{\sqrt{\eta_i}}2+
 \frac{\ell_{n_i}+\ell_{n_i+2}}4\longrightarrow0.
\]
The lifted laws $\widehat\nu_i$ also converge to $\mu_0$.
Entropy lower semicontinuity therefore yields
\[
 \Ent_m(\mu_0)\le\liminf_i H_{n_i+1}(\nu_i).
\]
For the projected endpoints, the same lifting argument and Jensen's
inequality give convergence of their entropies. Passing to the
midpoint inequality consequently yields
\[
 u(0)\ge S_{K/N}(2\tau c)\bigl[u(-\tau)+u(\tau)\bigr],
 \qquad u(t)=e^{-H(t)/N},\quad H(t)=\Ent_m(\mu_t).
\]
Expansion at $\tau=0$ gives $H''(0)-H'(0)^2/N\ge Kc^2$.
Since
\[
 H'(0)=-\alpha+\int\psi'v\,\dd\mu_0,\qquad
 H''(0)=\alpha^2+\int\psi''v^2\,\dd\mu_0,
\]
letting $a \to 0$ gives, for every $\alpha\in\R$,
\[
 K\le\psi''(x_0)+\alpha^2
       -\frac{(\psi'(x_0)-\alpha)^2}{N}.
\]
For $N>1$, minimize at $\alpha=-\psi'(x_0)/(N-1)$ and use continuity
at the endpoints to obtain $K\le K_*(N)$.
For $N=1$, varying $\alpha$ forces $\psi'\equiv0$, and then $K\le0$.
\end{proof}

\begin{remark}
By \eqref{eq:adaptive-weighted-convergence},
Theorem~\ref{thm:characterization} identifies the same entropic
curvature--dimension range on the limit interval; for $N>1$,
Corollary~\ref{cor:classical-cd} also gives the classical condition.
Three choices of density illustrate the result.
\begin{enumerate}
\item \emph{Uniform measure.} For $\psi\equiv0$, one has $K_*(N)=0$,
and every such refinement satisfies $\SCD(0,N)$ for $1\le N<\infty$.
Here $C_S=0$, so the entropy midpoint inequality is exact.
The masses are $|C_n(x)|$; on a uniform grid the endpoint masses
are half the interior masses.
\item \emph{Exponential weight.} For $\psi(x)=\beta x$, $\beta\ne0$,
the curvature bound is $K_*(N)=-\beta^2/(N-1)$ for $N>1$.
\item \emph{Centered Gaussian weight.} For
$\psi(x)=\tfrac12(x-\tfrac12)^2$, one has
$K_*(N)=1-[4(N-1)]^{-1}$. In particular, the sequence satisfies
$\SCD(3/4,2)$ and its limit satisfies $\mathrm{CD}(3/4,2)$.
\end{enumerate}
Thus the same interval metric and refinement rule admit zero,
negative, or positive curvature bounds through the reference measure.
\end{remark}

\subsection{Hypercube}
\label{sec:hypercube-refinements}

We now compare two refinements of hypercube.
Let $Q=[0,1]^d$ with $d \ge 2$.
Fix a sufficiently large $a_d$ that is a power of $2$.
Starting from $[0,1]$, divide each interval into $a_d$ equal
subintervals at every step. The resulting grid is
\[
 \varepsilon_{n}=a_d^{-n},\qquad I_n=\{0,\varepsilon_{n},\ldots,1\}.
\]
Take its coordinate product and retain the boundary points:
\[
 V_n=I_n^d\cap\partial Q.
\]
Thus a label belongs to $V_n$ exactly when at least one coordinate
is $0$ or $1$, and each facet cell splits into $a_{d}^{d-1}$ cells.
The two constructions use these same labels, with different
positions and connections.

\emph{Flat boundary refinement.}
Place each label $u$ at $u\in\partial Q$ and join nearest coordinate
neighbours on each facet by edges of length $\varepsilon_{n}$.
The path metric $d_n$ is the restriction of the intrinsic
$\ell^1$ metric $d_Q$ on $\partial Q$.

\emph{Projected boundary refinement.}
Let $c=(1/2,\ldots,1/2)$ and $R_d=\sqrt d/2$, and let $M$ be the
circumsphere of $Q$. Place the label $u$ at
\[
 P(u)=c+R_d\frac{u-c}{|u-c|}.
\]
The geometric vertex set is therefore $P(V_n)$, with
$|P(u)-c|=R_d$ for every vertex.
Subdivide the original parameter cells on $\partial Q$ and project
the new points; previously projected vertices keep their positions.
Set $r_n=\sqrt{\varepsilon_{n}}$ and join distinct labels $u,v$ whenever
$|P(u)-P(v)|\le r_n$, giving the edge this chord length.
Write $d_n$ for the resulting path metric.

For the masses, assign each point of a parameter facet to the grid
vertex obtained by rounding its free coordinates, with a fixed
convention for ties. Combine the regions assigned to the same vertex.
These are the flat control cells; their radial images give the
projected control cells. In either construction, denote the cells by
$C_n(u)$ and the cell-label map by $\pi_n$.
Let $\sigma$ be normalized facet volume or normalized spherical
area, respectively, so that the total reference mass is one.
Define
\begin{equation}\label{eq:hypercube-cell-masses}
 m_n(u)=\sigma(C_n(u)),\qquad
 \gamma_n(u,v)=\sigma(C_n(u)\cap C_{n+1}(v)).
\end{equation}
These overlap couplings have the prescribed marginals.

For comparison between layers, use $b_n(u,v)=d_Q(u,v)$ in
the flat construction. In the projected construction, join each
old label to its next-layer copy with length $\tau_n=T_dr_n$,
where $T_d$ is sufficiently large, and let $b_n$ be the resulting
cross-layer path distance.
Both rules are stationary and local, with complete resolution; 
the overlap couplings satisfy \eqref{eq:reference-displacement}.
Both initial spaces are the unit-edge hypercube graph on
$V_0=\{0,1\}^d$, with $m_0(u)=2^{-d}$.

\begin{proposition}
\label{prop:hypercube-refinements}
Let $\Sigma_{\mathrm{flat}}$ and $\Sigma_{\mathrm{sph}}$
be the refinement systems above, then we have
\[
 \Sigma_{\mathrm{flat}} \text{ satisfies } \SCD(0,d-1),
 \qquad
  \Sigma_{\mathrm{sph}} \text{ satisfies } \SCD(\frac{4(d-2)}{d}, d-1).
\]
\end{proposition}
\begin{proof}
We verify the midpoint condition for each refinement system 
and apply Theorem~\ref{thm:direct-midpoint}.

\emph{Flat construction.}
For $\eta>0$, let $d_\eta$ be the length metric on $X=\partial Q$
induced by
\[
 F_\eta(v)=\sum_{i=1}^d\sqrt{v_i^2+\eta|v|_2^2},\qquad
 d_Q\le d_\eta\le(1+d\sqrt\eta)d_Q.
\]
This norm is smooth, strongly convex, and invariant under signed
coordinate permutations. Thus $\mathcal U=X\setminus S$ is locally
Minkowski with constant reference density, where $S$ consists of
intersections of three or more facets.

Minimizers between regular points avoid $S$.
Indeed, near a $k$-fold intersection, $k\ge3$, the local model is
$\partial\mathbb R_+^k\times\mathbb R^{d-k}$.
For nearby points $p,r$ on opposite sides of the origin along a
minimizer, the intervening length is $F_\eta(p)+F_\eta(r)$.
If $p,r$ share a zero normal coordinate, their chord shortens the
path unless they are oppositely collinear, forcing all normal
coordinates to vanish.
Otherwise, choose $p_i=r_j=0$ and unfold the second facet by
$(r_i,0)\mapsto(0,-r_i)$.
Its chord folds back to a shorter path: equality forces the other
$k-2\ge1$ normal coordinates to vanish, contradicting this case.
Hence a minimizer meeting $S$ locally remains in its stratum;
its inverse image of $S$ is open and closed.

For finite-entropy endpoints compactly supported in $\mathcal U$,
the minimizing geodesics
have compact union $G\subset\mathcal U$.
Choose smooth $\lambda\ge1$, equal to one near $G$ and at least
$\operatorname{dist}_{\mathbb R^d}(\cdot,S)^{-1}$ near $S$
($\lambda=1$ if $S=\varnothing$).
Then $\lambda F_\eta$ is complete and preserves the relevant
distances and geodesics. Keep $\sigma$ unchanged.
The Finsler Jacobian comparison (see \cite{Ohta2009}), whose volume-distortion
coefficients here equal one, gives an optimal interpolation
$\mu_t=(T_t)_\#\mu_0=\rho_t\sigma$ satisfying
\[
 \rho_{1/2}(T_{1/2}x)^{-1/(d-1)}
 \ge\tfrac12\bigl[\rho_0(x)^{-1/(d-1)}
                  +\rho_1(T_1x)^{-1/(d-1)}\bigr].
\]
For $d = 2$, this is the one-dimensional Jacobian inequality on a
circle. The geometric mean
$\mathcal G(f)=\exp(\int\log f\,d\mu_0)$ satisfies
\[
 \mathcal G\bigl((f+g)/2\bigr)
 \ge\tfrac12\bigl[\mathcal G(f)+\mathcal G(g)\bigr],
\]
by logarithmic concavity with weights proportional to
$\mathcal G(f)$ and $\mathcal G(g)$. Consequently,
\begin{equation}\label{eq:hypercube-flat-entropy-midpoint}
 U_{d-1}(\mu_{1/2}\mid\sigma)
 \ge\tfrac12\bigl[U_{d-1}(\mu_0\mid\sigma)+U_{d-1}(\mu_1\mid\sigma)\bigr].
\end{equation}
Restricting and normalizing finite-entropy endpoints on compact
sets exhausting $\mathcal U$ gives convergence of their entropies.
Compactness of optimal midpoints and lower semicontinuity of entropy
extend \eqref{eq:hypercube-flat-entropy-midpoint} to arbitrary
finite-entropy endpoints, since $\sigma(S)=0$.

Now fix arbitrary endpoint laws $p_-,p_+$ on layers $n,n+2$.
Their uniform cell lifts satisfy
\[
 \Ent_\sigma(\widetilde p_-)=H_n(p_-),\qquad
 \Ent_\sigma(\widetilde p_+)=H_{n+2}(p_+).
\]
Take $\eta=\varepsilon_{n}^2/d^2$, choose the midpoint
$\widetilde\nu$ above, and set
$\nu=(\pi_{n+1})_\#\widetilde\nu$.
Writing $D=W(p_-,p_+)$ and
$\widetilde D=W_{2,d_\eta}(\widetilde p_-,\widetilde p_+)$,
the metric comparison and cell radii give
\[
 |D-\widetilde D|\le C\varepsilon_{n},\qquad
 W(p_-,\nu),\ W(\nu,p_+)\le\widetilde D/2+C\varepsilon_{n}.
\]
Since $H_{n+1}(\nu)\le\Ent_\sigma(\widetilde\nu)$ by Jensen,
\begin{align*}
 2W(p_-,\nu)^2+2W(\nu,p_+)^2&\le D^2+C\varepsilon_{n},\\
 U_{n+1}(\nu)&\ge\frac{1}{2}[U_n(p_-)+U_{n+2}(p_+)].
\end{align*}
Here and below the layer entropy powers use dimension $d-1$.
The constant is independent of $p_-,p_+$; since $\varepsilon_{n}\to0$,
this proves $\SCDmid(0,d-1)$.

\emph{Projected construction.}
The cells have radius $O_d(\varepsilon_{n})$, and $c_dr_n\le s_n(u)\le r_n$;
choosing $a_d$ large gives local resolution.
Write $\widehat d(u,v)=d_M(P(u),P(v))$.
Sampling geodesics and comparing arcs with chords gives
\[
 e_n:=\sup_{u,v\in V_n}|d_n(u,v)-\widehat d(u,v)|
 \le C_d(\varepsilon_{n}/r_n+r_n^2)=O_d(r_n).
\]
Taking $\tau_n\ge e_n+e_{n+1}$ preserves both layer metrics and,
for $u\in V_n$, $v\in V_{n+1}$, gives
\[
 \widehat d(u,v)\le b_n(u,v)
 \le\widehat d(u,v)+\tau_n+e_{n+1}.
\]
Lemma~\ref{lem:gluing}, the cell radii, and
$\sum_{k\ge n}\tau_k=O_d(r_n)$ yield
\begin{equation}\label{eq:hypercube-spherical-metric-error}
 |d_Z(\pi_i x,\pi_j y)-d_M(x,y)|\le C_dr_n,
 \qquad i,j\in\{n,n+1,n+2\}.
\end{equation}
They also bound all tail distances by $\pi R_d+C_dr_n$ and give
$b_n=O_d(r_n)$ on the support of $\gamma_n$.
Let $K_{d} = 4(d-2)/d$. 
Since $K_{d}(\pi R_d)^2 < (d-1) \pi^2$, a tail diameter bound
$L>\pi R_d$ lies in the comparison range.

The sphere has $\operatorname{Ric} = K_{d}g$ and satisfies the
finite-entropy midpoint comparison at dimension $d-1$ (see \cite{EKS2015}).
Apply the same cell lifting and midpoint projection on $M$.
Estimate~\eqref{eq:hypercube-spherical-metric-error} gives
$|D-\widetilde D|\le Cr_n$ and
$W(p_-,\nu),W(\nu,p_+)\le\widetilde D/2+Cr_n$,
where $\widetilde D=W_{2,M}(\widetilde p_-,\widetilde p_+)$.
Jensen and Lipschitz continuity of $S_{K_d/(d-1)}$ on $[0,L]$ give
\begin{align*}
 2W(p_-,\nu)^2+2W(\nu,p_+)^2&\le D^2+Cr_n,\\
 U_{n+1}(\nu)&\ge S_{K_d/(d-1)}(D)
       [U_n(p_-)+U_{n+2}(p_+)]-Cr_n.
\end{align*}
Since $r_n\to0$, these uniform estimates prove $\SCDmid(K_d,d-1)$.
\end{proof}

\subsection{Further examples}
\label{sec:resolution-boundaries}
\paragraph{Oscillating marked point}
We add to a dyadic grid a marked point of fixed positive mass.
Fix $0<w<1$ and, for $n\ge2$, set
\[
 \eta_n=2^{-n},\qquad G_n=\{k\eta_n:0\le k<2^n\},
\]
\[
 z_n=\frac{2+(-1)^n}{\sqrt2\,(n+1)},\qquad
 R_n=\frac{3}{\sqrt2\,(n+1)},\qquad P_n=[0,R_n].
\]
Define
\[
 V_n=G_n\cup\{z_n\},\qquad
 m_n=(1-w)\eta_n\sum_{g\in G_n}\delta_g+w\delta_{z_n}.
\]
Each grid point $g$ carries $[g,g+\eta_n]$; the anchor $0$ and
the marked point $z_n$ also carry $P_n$.
Their geometric sizes are $R_n$, and all other sizes are $\eta_n$.
Join consecutive grid points, join $z_n$ to its two bracketing
grid points and to $0$, and join $0$ to every nonzero grid point
in $P_n$. Give these connections their Euclidean lengths and use
Euclidean bridges, so that
\[
 d_n(x,y)=|x-y|,\qquad b_n(x,y)=|x-y|.
\]
The reference coupling sends half the mass at $g$ to each of
$g$ and $g+\eta_n/2$, and sends the marked mass to $z_{n+1}$.

The rule uses the marks
$(t,\varepsilon)=(1/(n+1),(-1)^n)$, updated by
$(t,\varepsilon)\mapsto(t/(1+t),-\varepsilon)$.
The stencil of a new grid point contains its old source and,
when that source lies in $P_n$, the anchor. The new anchor and
marked point use the old patch; the marked point also uses the
old cell containing its new position. These stencils have uniformly
bounded adjacency radius.

\begin{proposition}
\label{prop:endpoint-strictly-weaker}
\label{prop:atomic-limit-no-scd}
This construction is a complete refinement satisfying
\eqref{eq:reference-displacement}, with maximal scale $\ell_n=R_n$
and endpoint modulus $\omega_{n}=\eta_n+2R_n$.
The supported path $(z_n)$ has infinite total bridge length on every
tail. The measured limit is
\[
 ([0,1],|\cdot|,m),\qquad m=(1-w)\dd x+w\delta_0,
\]
which satisfies no $\CDe(K,N)$ condition for $K\in\mathbb R$
and $1\le N<\infty$. The refinement satisfies no $\SCD(K,N)$
condition for these parameters within the standing comparison range.
\end{proposition}

Endpoint control is compatible with infinite accumulated motion 
along a reference path. 
In this example, complete refinement ensures 
the existence of admissible sections, 
but the persistent mass carried by the marked point 
obstructs the entropy comparison for suitable endpoints. 
Thus complete refinement and endpoint control 
do not suffice for the SCD condition.

\paragraph{Cantor refinement}
Let $V_n$ consist of the left endpoints of the $2^n$ intervals
at stage $n$ of the middle-third Cantor construction.
Each representative carries its interval of length $3^{-n}$
and mass $2^{-n}$. Join consecutive representatives with their
Euclidean lengths, and use Euclidean bridges.
Each interval splits into its left and right thirds and divides
its mass equally; the representative displacements are
$0$ and $2\cdot3^{-n-1}$.

\begin{proposition}\label{prop:cantor-action-obstruction}
The cell diameters tend to zero, but complete resolution fails:
$\ell_n\ge1/3$ for $n\ge1$.
For
\[
 p=\delta_0\in\Pp(V_r),\qquad
 q=\delta_{c_r}\in\Pp(V_{r+2}),\qquad
 c_r=1-3^{-(r+2)},\qquad r\ge1,
\]
the midpoint action defect satisfies
\begin{equation}
 \min_{\nu\in\Pp(V_{r+1})}
 \bigl[2W(p,\nu)^2+2W(\nu,q)^2-W(p,q)^2\bigr]
 =\left(\frac13+3^{-(r+2)}\right)^2\ge\frac19.
 \label{eq:cantor-action-obstruction}
\end{equation}
\end{proposition}

Although the carried intervals shrink, 
the central gap keeps the midpoint action defect bounded away from zero. 
This explains why the geometric scale 
includes distances between adjacent elements as well as cell sizes. 
The construction lies outside the class of complete refinements: 
shrinking cells alone do not ensure the existence of admissible sections.

\paragraph{Goldberg--Coxeter subdivision for $C_{60}$}

$C_{60}$, known as the fullerene, is one of the trivalent graphs on sphere.
Its Goldberg--Coxeter subdivision has been well studied 
by Kotani, Naito and Tao in \cite{KotaniNaitoTao2022}.
Let $G_n=(V_n,E_n,F_n)$ be the trivalent surface graph and
$x_n:V_n\to\mathbb R^3$ its realization, starting from $C_{60}$ with
edge length one. For each incidence $v\in f$, introduce a vertex $v_f$,
and retain an inherited label $v^*$ for each old vertex. Join the
incidence vertices cyclically inside each old face and join $v^*$ to
its three incidence vertices. For $f=(v_0,\ldots,v_{k-1})$, solve
\begin{equation}
 3y_{f,v_i}-y_{f,v_{i-1}}-y_{f,v_{i+1}}=x_n(v_i)
 \qquad(i\bmod k),
 \label{eq:gc-short-face-rule}
\end{equation}
and set
\begin{equation}
 x_{n+1}(v_f)=y_{f,v},\qquad
 x_{n+1}(v^*)=\frac13\sum_{f\ni v}y_{f,v}.
 \label{eq:gc-short-relocation}
\end{equation}
These are the face Dirichlet and vertex relocation steps of the
geometric GC construction.

Give every edge $\{v,w\}$ the length $|x_n(v)-x_n(w)|$, and let $d_n$
be the resulting shortest-path metric. If $v_1,v_2,v_3$ are the
neighbours of $v$, define
\begin{equation}
 \begin{split}
 a_n(v)&=\frac12\left|
 (x_n(v_2)-x_n(v_1))\times(x_n(v_3)-x_n(v_1))\right|,\\
 A_n&=\sum_{v\in V_n}a_n(v),\qquad m_n(v)=a_n(v)/A_n.
 \end{split}
 \label{eq:gc-short-mass}
\end{equation}
We work with the resulting nondegenerate complete refinement,
whose layers $X_n=(V_n,d_n,m_n)$ are equipped with comparison data
satisfying the standing assumptions. 
For suitable endpoint measures, the action constraint forces
approximate midpoints into a region with too little reference mass.
The resulting entropy increase is incompatible with the midpoint
comparison, and the obstruction persists under further
refinement. Therefore we have 

\begin{proposition}[Midpoint obstruction for the GC sequence]
\label{prop:gc-short-obstruction}
Every comparison system $\Sigma$ for these layers satisfying the
assumptions above fails $\SCD(0,N)$ for every $1\le N<\infty$.
\end{proposition}

This example detects a curvature--dimension 
obstruction directly on successive refinement layers, 
without assuming the existence of a limiting metric-measure space. 
For the specified graph metrics and area weights, 
the obstruction to $\SCD(0,N)$ persists for every finite $N$. 
This is consistent with the negative curvature observed numerically in the GC subdivisions.

\section*{Acknowledgements}
The author is grateful to Professor Motoko Kotani
and Professor Hisashi Naito for helpful discussions
and valuable suggestions.
\section*{Declarations}
\paragraph{Funding.}
This work was supported by the National Natural Science
Foundation of China (Grant No.~12301064).

\end{document}